\documentclass[11pt]{amsart}

\usepackage{verbatim, amssymb,hyperref}

\usepackage{color}
\usepackage{bbm}
\usepackage{graphicx}
\usepackage{subcaption}
\usepackage{float}

\newcommand{\red}{\color[rgb]{1,0,0}}
\newcommand{\blue}{\color[rgb]{0,0,1}}

\newcommand{\orange}{\color[rgb]{0.83,0.38,0.22}}

\newcommand{\eps}{\varepsilon}

\newcommand{\cF}{\mathcal F}

\newcommand{\cD}{\mathcal D}

\newcommand{\cO}{\mathcal O}

\newcommand{\grad}{\nabla}

\newcommand{\sfrac}[2]{\mbox{$\frac{#1}{#2}$}}

\newcommand{\loja}{\L ojasiewicz}

\newcommand{\IB}{\mathbb B}

\newcommand{\1}{1\hspace{-0.098cm}\mathrm{l}}

\renewcommand{\P}{{\mathbb P}}

\newcommand{\N}{{\mathbb N}}
\newcommand{\IA}{{\mathbb A}}

\newcommand{\E}{{\mathbb E}}

\newcommand{\R}{{\mathbb R}}

\newcommand{\Hess}{\text{Hess}\,}

\theoremstyle{plain}
\newtheorem{theorem}{Theorem}[section]
\newtheorem{proposition}[theorem]{Proposition}
\newtheorem{lemma}[theorem]{Lemma}
\newtheorem{corollary}[theorem]{Corollary}

\theoremstyle{definition}
\newtheorem{remark}[theorem]{Remark}

\title[Convergence rates for momentum stochastic gradient descent]%
	{Exponential convergence rates for momentum stochastic gradient descent in the overparametrized setting}

\author[]
{Benjamin Gess}
\address{Benjamin Gess\\
	Fakult\"at f\"ur Mathematik\\
	Universit\"at Bielefeld\\
	Universit\"atsstrasse 25\\
	33615 Bielefeld\\
	Germany, 
	Max Planck Institute for Mathematics in the Sciences\\
	Inselstrasse 22\\
	04103 Leipzig\\ Germany}
\email{bgess@math.uni-bielefeld.de}
	
\author[]
{Sebastian Kassing}
\address{Sebastian Kassing\\
	Fakult\"at f\"ur Mathematik\\
	Universit\"at Bielefeld\\
	Universit\"atsstrasse 25\\
	33615 Bielefeld\\
	Germany}
\email{skassing@math.uni-bielefeld.de}
	

\begin{document}

\begin{abstract}
	We prove explicit bounds on the exponential rate of convergence for the momentum stochastic gradient descent scheme (MSGD) for arbitrary, fixed hyperparameters (learning rate, friction parameter) and its continuous-in-time counterpart in the context of non-convex optimization. In the small step-size regime and in the case of flat minima or large noise intensities, these bounds prove faster convergence of MSGD compared to plain stochastic gradient descent (SGD). The results are shown for objective functions satisfying a local Polyak-\L ojasiewicz inequality and under assumptions on the variance of MSGD that are satisfied in overparametrized settings. Moreover, we analyze the optimal choice of the friction parameter and show that the MSGD process almost surely converges to a local minimum.
\end{abstract}

\maketitle

\section{Introduction}
Many machine learning tasks involve the minimization of a function $f:\R^d \to \R$ given as an expectation $f(x)=\E[g(x,\Gamma)]$ for a random variable $\Gamma$ and a non-negative loss $g$. For example, in supervised learning one aims to minimize the average loss over a fixed training data set. 
In practice, the large size of the employed data sets requires the use of stochastic optimization methods, such as stochastic gradient descent (SGD). Such methods use random approximations of the gradient $\nabla f(x)$ for each iteration, e.g. through i.i.d.\ samples of $\nabla g(x,\Gamma)$. 

A second main challenge for the theoretical analysis of stochastic optimization algorithms in machine learning is the non-convexity of the loss landscape. In particular, often objective functions in supervised learning using neural networks possess rich, non-discrete sets of global minima, see e.g. \cite{cooper2018loss, fehrman2020convergence, dereich2022minimal}. 

Empirical observations~\cite{sutskever2013importance, gadat2018stochastic, sebbouh2021almost} motivate the long-standing conjecture that including momentum improves the performance of stochastic optimization algorithms. In recent years, a large class of  optimization algorithms has been proposed using combinations of various variants of momentum with other techniques such as adaptive step-sizes, preconditioning and batch-normalization \cite{nesterov1983method, QIAN1999145, duchi2011adaptive, kingma2014adam}. However, there are only few theoretical results proving the advantage of these methods. In fact, known results are restricted either to deterministic and continuous-in-time systems \cite{polyak1964some, aujol2022convergence, aujol2022convergence1, apidopoulos2022convergence}, or to deterministic systems with strongly convex objective functions~\cite{polyak1964some, ghadimi2015global}. For stochastic momentum algorithms, the available literature is bounded to qualitative statements~\cite{gadat2018stochastic, liu2023almost} and recovering the convergence rates found for SGD in the convex setting~\cite{gadat2018stochastic, sebbouh2021almost}.
This poses as an open problem the derivation of explicit bounds on the rate of convergence for time-discrete momentum stochastic gradient descent (MSGD) in a non-convex loss landscape, as it is met in machine learning. This problem is solved in the present work.

More precisely, we consider the MSGD algorithm 
\begin{align} \begin{split}\label{eq:MSGDapp}
		X_{n+1} &= X_n + \gamma_{n+1}V_{n+1}, \\
		V_{n+1}&= V_n - \gamma_{n+1}\mu V_n -\gamma_{n+1}\nabla g(X_n,\Gamma_{n+1}),
	\end{split}
\end{align}
for starting values $X_0,V_0 \in \R^d$, a sequence of strictly positive reals $(\gamma_n)_{n \in \N}$, a friction parameter $\mu>0$ and an i.i.d. sequence $(\Gamma_n)_{n \in \N}$
and derive explicit bounds on the exponential rate of convergence of  $(f(X_n))_{n \in \N_0}$.
 In the small step-size regime, these results rigorously justify the conjecture that the inclusion of momentum accelerates the convergence compared to SGD~\cite{wojtowytsch2021stochastic} for flat minima in overparametrized settings, that is, if $\min_{x \in \R^d} f(x)=0$.\footnote{See Section~\ref{sec:MLintro} for a discussion in the case of supervised learning}

In fact, we treat more general situations, including \eqref{eq:MSGDapp} as a special case: 
We assume throughout that $f: \R^d \to \R$ is a differentiable function with $C_L$-Lipschitz continuous gradient,\footnote{We comment on the necessity of this global Lipschitz continuity in Remark~\ref{rem:Lip1} and Remark~\ref{rem:Lip2} below.} for some constant $C_L \ge 0$, such that $\inf_{x \in \R^d} f(x)=0$. 
Let $(\Omega, (\mathcal F_{n})_{n \in \N_0}, \mathcal F, \P)$ be a filtered probability space and let $(X_n)_{n \in \N_0}$, $(V_n)_{n \in \N_0}$ be $(\mathcal F_n)_{n \in \N_0}$-adapted processes satisfying for all $n \in \N_0$
\begin{align}\begin{split} \label{eq:MSGDintro}
		X_{n+1} &= X_n + \gamma_{n+1}V_{n+1}, \\
		V_{n+1}&= V_n - \gamma_{n+1}\mu V_n -\gamma_{n+1}\nabla f(X_n) + \gamma_{n+1} D_{n+1},
	\end{split}
\end{align}
where $X_0, V_0 \in L^2(\Omega, \cF_0)$, $(\gamma_{n})_{n \in \N}$ is a sequence of strictly positive reals, $\mu>0$ and $(D_n)_{n \in \N}$ is a sequence of $L^2$-martingale differences with respect to the filtration $(\cF_n)_{n \in \N_0}$. In the following, we also call $(X_n)_{n \in \N_0}$ given by \eqref{eq:MSGDintro} the MSGD scheme with step-sizes $(\gamma_{n})_{n \in \N}$ and friction parameter $\mu$. The choice $(D_n)_{n \in \N}=(\nabla  f(X_{n-1})-\nabla g(X_{n-1},\Gamma_{n}))_{n \in \N}$ recovers the algorithm \eqref{eq:MSGDapp}.

	We state a simplified version of the main result in the case of constant step-sizes.
	
	\begin{theorem} \label{thm:intro1} (See Theorem~\ref{theo1} and Theorem~\ref{rem:constrained1}) Let $\gamma_n \equiv \gamma>0$. Let $L>0$ and $\sigma \ge 0$. Let $\mathcal D\subset \R^d$ be an open set and assume that for all $x \in \mathcal D$
		\begin{align} \label{eq:PLintro}
			|\grad f(x)|^2 \ge 2L f(x).
		\end{align}
		Moreover, for $n \in \N_0$, let
		$
		\IA_n = \{X_i \in \mathcal D \text{ for all } i =0, \dots, n\}
		$
		and assume that
		\begin{align} \label{eq:assunoise}
			\E[|D_{n+1}|^2 | \cF_n] \le \sigma f(X_n), \quad \text{ on } \IA_{n}.
		\end{align}
		If there exist parameters $a,b \ge 0$ such that all of the inequalities in \eqref{eq:constrained} are satisfied then:
		\begin{enumerate}
			\item[(i)] For all $\eps>0$ one has
			\begin{align*}
				\E[\1_{\IA_{n-1}} f(X_{n})] = o((r_{\operatorname{MSGD}}-\eps)^{-n}),
			\end{align*}
			where $r_{\operatorname{MSGD}}:= \min(1+a\gamma, \delta^{-1})$ and $\delta$ is given by \eqref{eq:constrained}.
			\item[(ii)] If $\delta<1$, the process $(X_n)_{n \in \N_0}$ converges almost surely on $\IA_\infty:= \bigcap_{n \in \N_0} \IA_n$.
		\end{enumerate}
		Moreover, for fixed $\mu$ and sufficiently small $\gamma$, there exist constants $a,b$ such that the above assumptions are satisfied and $\delta < 1$.
	\end{theorem}

Theorem~\ref{thm:intro1} provides a localized analysis of the rate of convergence for MSGD under two main assumptions: First, instead of a convexity assumption, we work with the local gradient inequality \eqref{eq:PLintro} which is often referred to as Polyak-\L ojasiewicz inequality (PL-inequality). Second, we assume that the variance of the stochastic perturbation vanishes as the process approaches a critical point. Section~\ref{sec:MLintro} below demonstrates that these assumptions are satisfied in overparametrized supervised learning.

	Note that the present setup is fundamentally different from other recent contributions~\cite{lessard2016analysis, yue2023lower, goujaud2023provable}. Theoretical results in optimization often compare the rate of convergence for the optimally chosen hyperparameters.
	It may be argued that in practice, an optimal choice of hyperparameters is impossible, since the problem parameters $L,C_L$ and $\sigma$ are unknown.
	Motivated from this we analyze MSGD for fixed hyperparameters. 
	In Remark~\ref{rem:SGDcomparison} below, we analyze the rigorous rates of convergence found in Theorem~\ref{thm:intro1} in a regime of step-sizes that is typically chosen as a default value.
	In order to ensure the robustness of the optimization, the step-size is often chosen to be small. Accordingly, we lay-out our findings in the small step-size regime and compare the convergence rate of MSGD derived in Theorem~\ref{thm:intro1} with the convergence rates for SGD.

Since the assumptions \eqref{eq:PLintro} and \eqref{eq:assunoise} are only assumed to hold locally, the convergence rates are conditioned on the event that the optimization dynamics stay inside $\mathcal D$. However, the estimates obtained in Theorem~\ref{thm:intro1} can be used to bound the probability of leaving this domain under the assumption that MSGD is initialized close to a critical point and with small initial velocity, see Corollary~\ref{cor:prob1}. Moreover, on the set $\IA_\infty = \bigcap_{n \in \N} \IA_n$ almost sure exponential convergence  of the objective function value to zero and of $(X_n)_{n \in \N}$ to a critical point is shown in Theorem~\ref{theo1}.

In contrast to qualitative convergence results, the derivation of  explicit bounds on the rate of convergence requires the careful selection of a suitable Lyapunov function, see \eqref{eq:Lyapunov} below, and the constrained optimization over hyperparameters, such as the friction parameter $\mu$, and additional technical parameters defining the Lyapunov function, see Lemma~\ref{lem:49378}. In addition, the localization of the assumptions in Theorem~\ref{theo1} relies on a detailed control of the event of leaving the domain $\mathcal D$, see e.g. \eqref{eq:2776362562} and Lemma~\ref{lem:Epositive}.

In the second part of this article, we investigate the continuous-in-time counterpart of the MSGD method. 
Assume that, additionally, $f$ is twice continuously differentiable and let $\Sigma: \R^d \to \R^{d \times d'}$ be a Lipschitz continuous function. Let $(\Omega, (\cF_t)_{t \ge 0}, \cF, \P)$ be a filtered probability space satisfying the usual conditions and consider the following system of SDEs
\begin{align}
	\begin{split} \label{eq:SDEintro}
		dX_t &= V_t \, dt,\\
		dV_t &= - (\mu V_t + \grad f(X_t) ) \, d t +  \Sigma(X_t) \, d W_t,
	\end{split}
\end{align}
where $V_0,X_0 \in L^4(\Omega, \cF_0)$, $\mu >0$ and $(W_t)_{t \ge 0}$ is a standard $\R^{d'}$-valued $(\cF_t)_{t \ge 0}$-Brownian motion.

The Lipschitz continuity of $\grad f$ and $\Sigma$ imply that there exists a unique continuous $\R^{2d}$-valued semimartingale $(X_t,V_t)_{t \ge 0}$ satisfying \eqref{eq:SDEintro}. Moreover, for all $T \ge 0$ there exists a constant $C>0$ such that
$
\E[\sup_{t \in [0,T]} (|X_t|^4 + |V_t|^4)] < C(1+ \E[|X_0|^4 + |V_0|^4]),
$
see e.g. Theorem~19 in \cite{li2019stochastic}, so that $\grad f(X_t) \in L^4(\Omega)$ and $f(X_t) \in L^2(\Omega)$, for all $t \ge 0$. We show the exponential convergence of $(f(X_t))_{t \ge 0}$ for an objective function $f$ that satisfies the PL-condition in an open set $\mathcal D$. For a properly chosen friction parameter $\mu$, we estimate the influence of the fluctuations on the optimal rate of convergence, and compare to the one derived for the heavy-ball ODE in \cite{apidopoulos2022convergence}. For a comparison of the convergence rate for the system \eqref{eq:SDEintro} and a continuous-in-time version of SGD \eqref{eq:SDESGD} we refer the reader to Remark~\ref{rem:compSDE}.

\begin{theorem}(See Theorem~\ref{theoSDE2}) \label{theointro}
	Let $L>0$, $C_L^* = C_L \vee \frac 98 L$ and $0<\sigma<4 \frac{L}{\sqrt{C_L^*}}$. Let $\mathcal D \subset \R^d$ be an open set such that for all $x \in \mathcal D$
	\begin{align*}
		|\grad f(x)|^2 \ge 2L f(x) \quad \text{ and } \quad  \|\Sigma(x)\|_F^2 \le \sigma f(x).
	\end{align*}
	and choose
	$$
	\mu = 2\sqrt{C_L^*} - \sqrt{C_L^* - L + \frac 14 \sqrt{C_L^*} \sigma}.
	$$
	Then, there exists a $C\ge 0$ such that
	$$
	\E[\1_{\{T>t\}} f(X_t)]\le C \exp(-mt), \quad \text{ for all } t\ge 0,
	$$
	where $T=\inf\{t \ge 0: X_t \notin \mathcal D\}$ and
	$$
	m= 2\Biggl(\sqrt{C_L^*} -\sqrt{C_L^*-L+\frac 14 \sqrt{C_L^*}\sigma}\Biggr).
	$$
\end{theorem}

\textbf{Overview of the literature:} 
Convergence rates for the solution to the heavy-ball ODE, i.e.
\begin{equation}\label{eq:HBODE}
	\dot x_t  = v_t ,\quad 		\dot v_t = - \mu v_t - \grad f(x_t),
\end{equation}
with $\mu >0$, have been derived in the literature under various assumptions on the loss landscape, starting from the work by Polyak~\cite{polyak1963gradient}. Polyak showed that, for $L$-strongly convex and twice differentiable functions $f$, $(f(x_t))_{t \ge 0}$ converges with rate $\mu-\sqrt{\max(0,\mu^2-4L)}$. The choice $\mu = 2 \sqrt L$ leads to a convergence rate of $2 \sqrt L$.
In comparison, for the solution to the gradient flow ODE, i.e.
\begin{equation} \label{eq:GFODE}
	\dot y_t = - \grad f(y_t),
\end{equation}
one has exponential convergence of $(f(y_t))_{t \ge 0}$ with rate $2L$.
Thus, choosing the optimization dynamics \eqref{eq:HBODE} instead of \eqref{eq:GFODE} is beneficial for objective functions $f$ that are comparatively flat around the global minimum.
In 1963, Polyak \cite{polyak1963gradient} and \L ojasiewicz \cite{lojasiewicz1963propriete} independently proposed the gradient inequality \eqref{eq:PLintro} which is a relaxation of the strong convexity assumption. It turns out  that \eqref{eq:PLintro} together with a Lipschitz assumption on the gradient of $f$ is still sufficient to prove the exponential convergence of $(f(y_t))_{t \ge 0}$ for solution to the gradient flow \eqref{eq:GFODE} and $(f(x_t))_{t \ge 0}$ for the solution to the heavy-ball ODE \eqref{eq:HBODE}, see \cite{polyak2017lyapunov}. The proof for the latter result relies on a Lyapunov function that contains the sum of the potential and kinetic energy of the dynamical system, as well as a cross-term of the two. \cite{aujol2022convergence1} obtains a convergence rate of $\sqrt{2L}$ for the friction parameter $\mu = 3 \sqrt{L/2}$ in the setting of $L$ quasi-strongly convex functions with Lipschitz continuous gradient having a unique isolated minimum. Moreover, they show that for every parameter $\mu<3\sqrt{L/2}$ there exists an $L$-strongly convex objective function $f$ (having only a H\"older continuous gradient) such that $(f(x_t))_{t \ge 0}$ converges at most with rate $\frac 23 \mu$. Note that quasi-strong convexity implies the PL-inequality, see \cite{aujol2022convergence}. In \cite{apidopoulos2022convergence}, an exponential rate of convergence for functions satisfying the PL-inequality is derived, proving a similar advantage of the heavy-ball dynamics over the gradient flow dynamics to the one found for flat, strongly convex functions. \cite{Cabot2007Onthelong} considered the heavy-ball ODE with time-dependent friction parameter. They give sufficient and necessary conditions for the decay rate of the friction in order get convergence of the process, as well as the $f$-value of the process, for convex objective functions.

Recently, the PL-inequality gained a considerable amount of attention due to its simplicity, its strong implications on the geometry and its applicability for objective functions appearing in machine learning, see e.g. \cite{karimi2016linear, dereich2021convergence, aujol2022convergence, kuruzov2023gradient, garrigos2023square, rebjock2023fast, wojtowytsch2021stochastic}.

For the discrete-in-time heavy ball scheme the situation is much more intricate. One needs to distinguish two fundamentally different problem setups: First, rates of convergence for optimally chosen hyperparameters, second, rates of convergence for arbitrary fixed hyperparameters. Regarding the first class, the seminal work by Polyak~\cite{polyak1964some} proves faster convergence of the deterministic heavy ball method compared to gradient descent when optimizing a \emph{quadratic function}
	and choosing the optimal parameters $\gamma, \beta >0$.
	Conversely, the counterexamples presented in~\cite{lessard2016analysis, goujaud2023provable} show that heavy ball does not accelerate on the much larger class of strongly convex objective functions for optimally chosen step-size. Moreover, in \cite{yue2023lower} it is proved that no first order method accelerates on the class of objective functions satisfying the PL-inequality with parameter $L$ for optimally chosen step-size.
	Nevertheless, the work \cite{danilova2020non} finds parameters $\gamma$ and $\beta$ such that heavy ball recovers the the best possible convergence rate of gradient descent on the class of PL-functions. 
	
	In this work, we consider the second fundamentally different situation, namely, the stochastic gradient and small step-size setting. We show that MSGD accelerates convergence for conservatively chosen step-sizes, i.e. in the small step-size regime, when converging to flat minima, as well as for large noise intensities, see Remark~\ref{rem:SGDcomparison}. As pointed out above, in general the constants $C_L,L$ and $\sigma$ are not known and the practitioner chooses a sufficiently small (and time-decreasing) step-sizes to at least guarantee convergence.

Note that the MSGD process is a slight variation of the stochastic heavy-ball (SHB), which generalizes Polyak's heavy-ball method by adding stochastic noise. According to~\cite{gadat2018stochastic}, the SHB process is defined via the iteration scheme
\begin{equation} \begin{split}\label{eq:SHBGadatref}
		X_{n+1} &= X_n + \gamma_{n+1}V_{n+1}, \\
		V_{n+1}&= V_n - \gamma_{n+1}\mu (\nabla f(X_n)+V_n-D_{n+1}).
	\end{split}
\end{equation}
It can be shown that \eqref{eq:SHBGadatref} is a discretization of \eqref{eq:HBODE} with an additional perturbation, where one iteration step with step-size $\gamma_n$ corresponds to the position of \eqref{eq:HBODE} after time $\sqrt{\gamma_n/\mu}$. Thus, compared to the immediate time discretization executed in the MSGD scheme \eqref{eq:MSGDintro} the SHB process \eqref{eq:SHBGadatref} speeds up the corresponding ODE time for small step-sizes. A similar phenomenon occurs in Nesterov acceleration. In~\cite{even2021continuized} the authors propose a continuized process using exponential stopping times so that no additional time change is needed in order to be able to compare the discrete process with the corresponding continuous-in-time counterpart. Convergence rates for the SHB in the convex setting can be found in~\cite{gadat2018stochastic, sebbouh2021almost}. In particular, \cite{gadat2018stochastic} recovers the optimal $\cO(1/n)$-convergence rates in the underparametrized regime for a broader class of step-sizes compared to SGD~\cite{RM51}, an effect also know for Ruppert-Polyak averaging~\cite{dereich2023central}.  \cite{loizou2017linearly,loizou2020momentum} derives an (accelerated) exponential convergence rate for SHB for solving a linear system with a random norm. In this setting, the stochastic gradient vanishes as SHB approaches the optimal point which is comparable to our assumption \eqref{eq:assunoise}.
Similar to SGD, SHB is able to avoid strict saddle points~\cite{liu2023almost} and converges on analytic objective functions under classical noise assumptions~\cite{dereich2021convergence}.

In~\cite{li2019stochastic} it has been shown that for an appropriately chosen diffusion matrix $\Sigma$ the SDE \eqref{eq:SDEintro} is a weak approximation of the MSGD process on a finite time interval. For the continuous-in-time counterpart of SGD, Wojtowytsch~\cite{wojtowytsch2021stochasticcont} showed that the special structure of the noise in overparametrized settings induces a tendency for the process to choose a flat minimum. Flat minima are commonly believed to generalize better, see e.g.~\cite{keskar2016large} for numerical experiments on the generalization gap and the sharpness of minima. In the mean-field scaling, the SGD dynamics have been shown to converge to solutions of conservative stochastic partial differential equations, see~\cite{gess2022conservative, gess2023stochastic}.
Hu et al.~\cite{hu2019global} investigated the behavior of an SDE similar to the one defined in \eqref{eq:SDEintro} near strict saddle points.

{\bf The paper is organized as follows:} In Section~\ref{sec:MLintro}, we motivate the assumptions on the objective function and the size of the stochastic noise from overparametrized supervised learning. Section~\ref{sec:MSGD} is devoted to the proofs of the results on the MSGD process in discrete time. In Section~\ref{sec:MSDE}, we prove the results on the continuous-in-time counterpart defined in \eqref{eq:SDEintro}.

\textbf{Notation:} We denote by $v^\dagger$ the transpose of a vector $v \in \R^d$, by $A^\dagger$ the transpose of a matrix $A \in \R^{n \times k}$ and by $\|A\|_F$, respectively $\|A\|$, the Frobenius norm, respectively operator norm of $A$. Moreover,  $|\cdot|$ denotes the standard Euclidean norm and $\langle \cdot, \cdot \rangle$ the standard scalar product on the Euclidean space.

\section{Loss landscape and noise in empirical risk minimization} \label{sec:MLintro}
In this section, we motivate the main assumptions on the loss landscape and the stochastic noise in a machine learning application. In particular, we consider a regression problem in supervised learning with quadratic loss function. Let $(\theta_1,\zeta_1), \dots, (\theta_N, \zeta_N) \in \R^{d_{\mathrm{in}}}\times \R^{d_{\mathrm{out}}}$ be a given training data set. We choose a parameterized hypotheses space $\mathcal S :=\{\mathfrak N^x(\cdot): x \in \R^{d}\}$ consisting of functions $\mathfrak N^x(\cdot): \R^{d_{\mathrm{in}}} \to \R^{d_{\mathrm{out}}}$ such that, for all $i=1, \dots, N$, $x \mapsto \mathfrak N^x(\theta_i)$ is differentiable.
For example, one can choose $\mathcal S$ to be the space of response functions of fully connected feed-forward neural networks with fixed architecture.
The aim of risk minimization (with respect to the square loss) is to select a suitable model $\mathfrak N^x(\cdot )$ minimizing the empirical risk
$$
f(x)=\frac 1{2N} \sum_{i=1}^N  |\mathfrak N^x(\theta_i)-\zeta_i|^2, \quad x \in \R^{d}.
$$
In order to derive a dynamical system as in \eqref{eq:MSGDintro} we choose deterministic starting values $X_0, V_0 \in \R^d$, a sequence of strictly positive reals $(\gamma_{n})_{n \in \N}$, an i.i.d. sequence $(I_n)_{n \in \N}$ such that $I_n$ is uniformly distributed on $\{1, \dots, N\}$ and consider the dynamical system
\begin{align*}
	X_{n+1} &= X_n + \gamma_{n+1}V_{n+1}, \\
	V_{n+1}&= V_n - \gamma_{n+1}\mu V_n -  \frac 12 \gamma_{n+1}   \nabla \bigl(|\mathfrak N^{x}(\theta_{I_{n+1}})-\zeta_{I_{n+1}}|^2\bigr)\big|_{x = X_n}.
\end{align*}
We recover \eqref{eq:MSGDintro} by choosing $$
D_{n+1} = \grad f(X_n)- \frac 12    \nabla \bigl(|\mathfrak N^{x}(\theta_{I_{n+1}})-\zeta_{I_{n+1}}|^2\bigr)\big|_{x = X_n}.
$$
We set $(\cF_n)_{n \in \N_0} = (\sigma(I_1, \dots, I_n))_{n \in \N_0}$ and note that, for all $n \in \N$, $\E[D_{n+1} | \cF_n]=0$ and 
\begin{align*}
	\E[&|D_{n+1}|^2 | \cF_n] \le C \sum_{i=1}^N  |(\mathfrak N^{X_n}(\theta_{i})-\zeta_{i})\grad \mathfrak N^{x}(\theta_{i})\big|_{x = X_n}|^2,
\end{align*}
for a constant $C>0$. On a domain $\mathcal D \subset \R^{d}$ where the gradient $ \grad_x \mathfrak N^x (\theta_i)$ is bounded for all $i=1, \dots, N$, this implies that
$$
\E[|D_{n+1}|^2 | \cF_n] \le \sigma f(X_n),
$$
for a constant $\sigma\ge 0$.
Analogously, the gradient $\grad f$ satisfies $
|\grad f(x)|^2 \le C f(x)
$, i.e. the inverse PL-inequality, for a constant $C \ge 0$ on the same domain $\mathcal D$.

We next motivate the PL-inequality. The regression problem is called overparametrized if there exists a $y \in \R^{d}$ with $f(y)=0$. The following result was shown by Cooper~\cite{cooper2018loss} for overparametrized regression problems satisfying $d>Nd_{\mathrm{out}}$ and $\mathfrak N^{\cdot}(p)$ being $C^{k}$-smooth for a $k \ge d-Nd_{\mathrm{out}}+1$ and all $p \in \R^{d_{\mathrm{in}}}$: for almost all tuples of training data (up to a Lebesgue nullset) the set of global minima $\mathcal M := \{x \in \R^{d}: f(x)=0\}$ forms a closed $(d-Nd_{\mathrm{out}})$-dimensional $C^k$-submanifold of $\R^{d}$. If such $\mathcal M$ is a $C^2$-manifold and, for a $y \in \mathcal M$, we have $\dim(\Hess f(y))=Nd_{\mathrm{out}}$, Theorem~2.1 of \cite{feehan2019resolution} shows that there exists a neighborhood $U \subset \R^{d}$ of $y$ such that a PL-inequality holds on $U$, i.e. there exists an $L>0$ with $2L f(x) \le |\grad f(x)|^2$ for all  $x \in U$.

The last result of this section is a general version of the inverse PL-inequality for functions $f: \R^d \to \R$ having a Lipschitz continuous gradient. This observation has already been made in \cite{wojtowytsch2021stochastic}, see Lemma~B.1 therein. We weaken the assumptions by only assuming local Lipschitz continuity on a ball around a local minimum. We will use this lemma repeatedly in the subsequent sections. 

\begin{lemma} \label{rem:Lipschitz}
	Let $r>0$, $y \in \R^d$ and assume that $\grad f$ is $C_L$-Lipschitz continuous on $B_r(y)$ and $\inf_{x \in B_r(y)}f(x)=f(y)$. Then, for all $x \in B_{r/2}(y)$ it holds that
	\begin{align} \label{eq:inversePL}
		|\grad f(x)|^2 \le 2 C_L (f(x)-f(y)).
	\end{align}
\end{lemma}

\begin{proof}
	Since $y$ is a critical point of $f$ we have for all $x \in B_{r/2}(y)$
	$$
	|\grad f(x)| = |\grad f(x)-\grad f(y)| \le \frac{C_L r}{2}.
	$$
	If $\grad f(x)=0$ the statement is obviously true. 
	If $\grad f(x) \neq 0$ consider the function 
	$$
	g(t) = f\Bigl(x-t \frac{\grad f(x)}{|\grad f(x)|}\Bigr).
	$$
	Note that for $x \in B_{r/2}(y)$ and all $t \in [0, \frac{|\grad f(x)|}{C_L}]$ we have $x-t \frac{\grad f(x)}{|\grad f(x)|} \in B_r(y)$ so that with the Lipschitz continuity of $\grad f$ and since $y$ is a local minimum
	\begin{align*}
		f(y)-f(x) &\le g\Bigl(\frac{|\grad f(x)|}{C_L}\Bigr)-g(0) = \int_0^{\frac{|\grad f(x)|}{C_L}} g'(s)  \, ds \\
		& \le \frac{|\grad f(x)|}{C_L} g'(0) +\frac{|\grad f(x)|^2}{2C_L}= -\frac{|\grad f(x)|^2}{2C_L}.
	\end{align*}
\end{proof}

\begin{remark}
	For functions $f:\R^d \to \R$ with $C_L$-Lipschitz continuous gradient satisfying the PL-inequality we get with Lemma~\ref{rem:Lipschitz} that 
	\begin{align} \label{eq:18376655554}
		2L (f(x)-f(y))\le |\grad f(x)|^2 \le 2 C_L (f(x)-f(y)),
	\end{align}
	where $x \in \R^d$ and $y \in \R^d$ is a global minimum of $f$ with $f(y)=0$. Thus, we immediately get $C_L \ge L$.
	In the strictly convex case
	$$
	f(x) =\frac 12 x^\dagger A x,
	$$
	for a positive definite matrix $A \in \R^{d \times d}$, the constants $C_L$, respectively $L$, in \eqref{eq:18376655554} correspond to the largest, respectively smallest, eigenvalue of $A$.
\end{remark}

\section{Momentum stochastic gradient descent in discrete time} \label{sec:MSGD}
In this section, we consider the MSGD scheme $(X_n)_{n \in \N_0}$ introduced in \eqref{eq:MSGDintro}. We state the main results of this section. First, we show exponential convergence of the objective function value in the numerical time $(t_n)_{n \in \N_0}=(\sum_{i=1}^n \gamma_i )_{n \in \N_0}$ for sufficiently small step-sizes. This implies almost sure convergence of the MSGD process itself.

\begin{theorem} \label{theo1}
	Let $L>0, \sigma \ge 0$.
	Let $\mathcal D\subset \R^d$ be an open set and assume that 
	$$
	|\grad f(x)|^2 \ge 2L f(x)
	$$
	for all $x \in \mathcal D$.
	Moreover, for $n \in \N_0$, let
	$
	\IA_n = \{X_i(\omega) \in \mathcal D \text{ for all } i =0, \dots, n\}
	$
	and assume that
	$$ 
	\E[|D_{n+1}|^2 | \cF_n] \le \sigma f(X_n), \quad \text{ on } \IA_{n}.
	$$
	There exists $\bar \gamma >0$ such that 
		if $\sup_{n \in \N}\gamma_n\le \bar \gamma$ there holds:
	\begin{enumerate}
		\item[(i)] There exist $C,m>0$ such that for all $n \in \N$ we have
		$$
		\E[\1_{\IA_{n-1}} f(X_n)] \le C \exp(-m t_n),
		$$
		where $t_n = \sum_{i=1}^n \gamma_i$.
		\item[(ii)] Let $m'<m$ and assume that $ \sum_{i=0}^\infty \exp((m'-m) t_i)<\infty$. Then, on $\IA_\infty= \bigcap_{n \in \N_0} \IA_n$, we have $\exp(m' t_n) f(X_n) \to 0$ almost surely.
		\item[(iii)] The process $(X_n)_{n \in \N_0}$ converges almost surely on $\IA_\infty$.
	\end{enumerate}
\end{theorem}

For step-sizes $(\gamma_n)_{n \in \N}$ with $\gamma_n \to 0$, there exists $N \in \N$ such that $\sup_{n > N} \gamma_n$ is sufficiently small in order to apply Theorem~\ref{theo1} for the system $(X_n)_{n \ge N}$ started at time $N$. However, Theorem~\ref{theo1} is also applicable for a constant sequence of step-sizes $\gamma_n \equiv \gamma$, as long as $\gamma$ is sufficiently small. Note that, since $X_0,V_0 \in L^2(\Omega, \mathcal F_0)$, the Lipschitz continuity of $\grad f$ and the assumptions on the process $(D_n)_{n \in \N}$ imply that for all $n \in \N$ we have $
\1_{\IA_{n-1}}X_n$, $\1_{\IA_{n-1}}V_n$, $\1_{\IA_{n-1}}\grad f(X_n) \in L^2(\Omega)$ and $\1_{\IA_{n-1}}f(X_n) \in L^1(\Omega)$. The convergence rate $m$ in Theorem~\ref{theo1} depends on $L, C_L, \sigma, \mu$ and $\sup_{n \in \N}\gamma$. 

Next, we
optimize $\mu$ over the set of friction parameters in the small step-size regime. We recover the convergence rates  for the heavy ball ODE~\eqref{eq:HBODE} derived in \cite{apidopoulos2022convergence} in terms of the numerical time $t_n = \sum_{i=1}^n \gamma_i$. Since comparison results for MSGD~\eqref{eq:MSGDintro} and the heavy ball ODE \eqref{eq:HBODE} on non-convex objective functions only hold on a finite time interval, see e.g.~\cite{wumean, li2019stochastic, gadat2022asymptotic}, the time continuous result does not carry over to the discrete-in-time setting. In fact, in the analysis of MSGD additional error terms appear due to the discrete nature. Therefore, the proof requires a worst-case analysis bounding these error terms over the set of allowed step-sizes.
	Theorem~\ref{theo2} motivates the comparison of MSGD and SGD in the small-learning rate regime, see Remark~\ref{rem:SGDcomparison}.

\begin{theorem} \label{theo2}
	Set $\kappa = \frac{C_L}{L}$. Let 
	\begin{align*}
		\mu \in 
		\begin{cases}
			\bigl[\frac{1}{\sqrt 8}\bigl(5-\sqrt{ 9 - 8\kappa}\bigr)\sqrt{L},\frac{1}{\sqrt 8}\bigl(5+\sqrt{ 9 - 8\kappa}\bigr)\sqrt{L} \, \bigr], & \text{ if } \kappa < \frac 98, \\
			\bigl\{(2\sqrt \kappa-\sqrt{\kappa-1})\sqrt L \bigr\}, & \text{ if } \kappa \ge \frac 98.
		\end{cases}
	\end{align*}
	Then, under the assumptions of Theorem~\ref{theo1}, for every $\eps>0$ there exist $C, \bar \gamma\ge 0$ such that if $\sup_{n \in \N}\gamma_n \le \bar \gamma$ it holds that
	\begin{align*}
		\E[\1_{\IA_{n-1}} f(X_n)] \le C \exp(-(m-\eps)t_n), \quad \text{ for all } n \in \N,
	\end{align*}
	where  
	$$
	m = 
	\begin{cases} \sqrt{2L},& \text{ if } \kappa<\frac{9}{8},\\
		2(\sqrt{\kappa}-\sqrt{\kappa-1})\sqrt{L},& \text{ if } \kappa \ge \frac 98.
	\end{cases}
	$$
\end{theorem}

Using estimates from the proof of Theorem~\ref{theo1}, we can bound the probability that $(X_n)_{n \in \N_0}$ leaves the domain $\mathcal D$ if it is initialized close to a global minimum and with small initial velocity. 

\begin{corollary} \label{cor:prob1}
	Let $y \in \cD$ with $f(y)=0$. Then, under the assumptions of Theorem~\ref{theo1}, for every $\eps>0$ there exists an $r_0>0$ such that if $X_0 \in B_{r_0}(y)$, almost surely, and $\E[|V_0|^2]\le r_0$ we have
	$$
	\P(\IA_\infty^c ) \le \eps.
	$$
\end{corollary}

	Our results are based on the following theorem that derives the exponential rate of convergence conditioned on solving a constrained optimization problem. 
	
	\begin{theorem} \label{rem:constrained1} Let $\gamma_n \equiv \gamma$ for a $\gamma >0$. Let $a, b \ge 0$ and assume that
		\begin{align} \begin{split} \label{eq:constrained}
				0 \ge & - 1 +\gamma \Bigl(\frac b2-a\Bigr) + \gamma^2 C_L + \gamma^3 \frac{C_L a}{2},\\
				0 \ge &a \mu+ ab-a^2-2L +\gamma \Bigl( \frac{b \sigma }{2}+a(2C_L-\mu b+\mu a) -2L(a-\frac b2) \Bigr)  \\
				&+ \gamma^2 C_L \Bigl( \sigma +a^2-2a\mu +2L \Bigr)  + \gamma^3 C_L a \Bigl( \frac{\sigma}{2}-a \mu+L  \Bigr),\\
				0 \ge &C_L - \frac{b}{2} (\mu+a-b) +\gamma \Bigl( \frac{b\mu^2}{2} + \frac{C_L a}{2}-2C_L\mu +\frac{ba\mu}{2}-\frac{b^2\mu}{2}+C_Lb \Bigr) \\
				&+\gamma^2 C_L \Bigl(  \mu^2-a\mu +\frac{ba}{2}-b\mu \Bigr)  + \gamma^3 \frac{C_La\mu}{2} ( \mu-b ) ,\\
				0 \le & \delta:= 1+\gamma(a-\mu-b)+ \gamma^2(b\mu-a\mu-2C_L)+\gamma^3(2C_L\mu-C_La)+\gamma^4 C_La\mu,\\
				& \qquad \qquad \qquad \quad ab \ge  C_L \quad \text{ and } \quad \gamma \le \frac{b}{C_L}.
			\end{split}
		\end{align}
		Then, under the assumptions of Theorem~\ref{theo1} there exists a constant $C \ge 0$ such that
		$$
		\E[\1_{\IA_{n-1}} f(X_{n})] \le \begin{cases}
			C (1+a\gamma)^{-n}, & \text{ if } 1+a\gamma < \delta^{-1} \\
			C \delta ^n, & \text{ if } 1+a\gamma > \delta^{-1} \\
			C(1+a\gamma)^{-n}n, & \text{ if } 1+a\gamma = \delta ^{-1}
		\end{cases}, \quad \text{ for all } n \in \N.
		$$
	\end{theorem}

	For fixed $\mu >0$ and sufficiently small $\gamma$ one can choose parameters $a,b > 0$ such that all inequalities above are satisfied and $\delta<1$. This will be made precise in the forthcoming analysis, see Proposition~\ref{prop:convrate} and Lemma~\ref{lem:feasible1}.

\begin{remark} \label{rem:SGDcomparison} 
	We compare the convergence rate for MSGD proven in Theorem~\ref{rem:constrained1} to the convergence rate for SGD found in \cite{wojtowytsch2021stochastic} which agrees with the results in \cite{karimi2016linear} in the noiseless case $\sigma=0$ (see also~\cite{vaswani2019fast, khaledbetter}).
	Theorem~\ref{rem:constrained1} shows that for all $\eps>0$ one has
	\begin{align} \label{eq:rate}
		\limsup_{n \to \infty } (r_{\operatorname{MSGD}}-\eps)^n \E[\1_{\IA_{n-1}} f(X_{n})] = 0,
	\end{align}
	where $r_{\operatorname{MSGD}}$ is the maximal value of $r(a,b):= \min(1+a\gamma, \delta^{-1})$ for all $a,b \ge 0$ such that \eqref{eq:constrained} holds.
	\cite{wojtowytsch2021stochastic} gives a convergence rate for SGD under the same assumptions on the objective function and the stochastic noise, in the sense of \eqref{eq:rate}, of $r_{\operatorname{SGD}}=1-2L\gamma + \gamma^2 \frac{C_L(2L+\sigma)}{2}$ for all step-sizes satisfying $\gamma < \frac{2L}{2L+\sigma} \frac{2}{C_L}$. 
	
	First, we fix $\gamma = 0.01$, which is a popular default value for the step-size \cite{bengio2012practical}, and compare the rate $r_{\operatorname{SGD}}$ for SGD with the rate $r_{\operatorname{MSGD}}$ for MSGD with optimally chosen friction parameter $\mu^*$ in the noiseless case (Figure~\ref{fig1}), as well as for high noise intensity (Figure~\ref{fig2}). 
	
	We observe that in the noiseless case $\sigma=0$, MSGD outperforms SGD for flat objective functions, i.e. for small $L$. For high noise intensity $\sigma =100$, MSGD is more robust. While SGD converges only when the condition number is small ($\kappa < 4 $), MSGD can adapt to the noise intensity and converges with an exponential rate in all given scenarios.
	
	\begin{figure}[H]
		\centering
		\begin{subfigure}{0.32\textwidth}
			\centering
			\includegraphics[width=\linewidth]{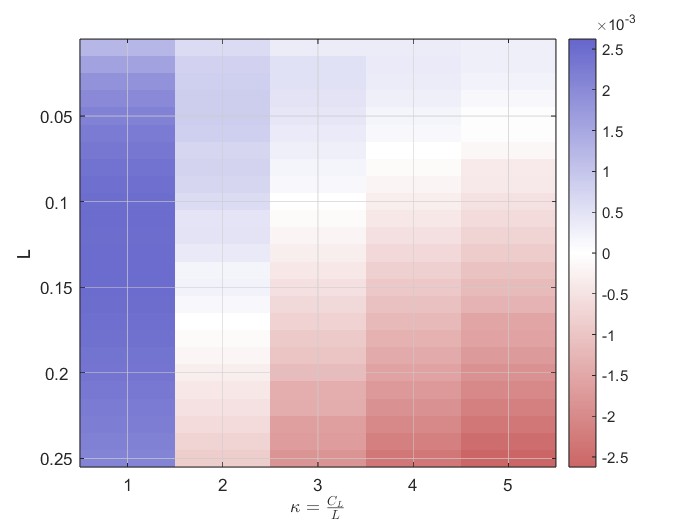}
			\caption*{$r_{\text{MSGD}}-r_{\text{SGD}}$\\
				\ }
			\label{fig1:graph1}
		\end{subfigure}
		\hfill
		\begin{subfigure}{0.32\textwidth}
			\centering
			\includegraphics[width=\linewidth]{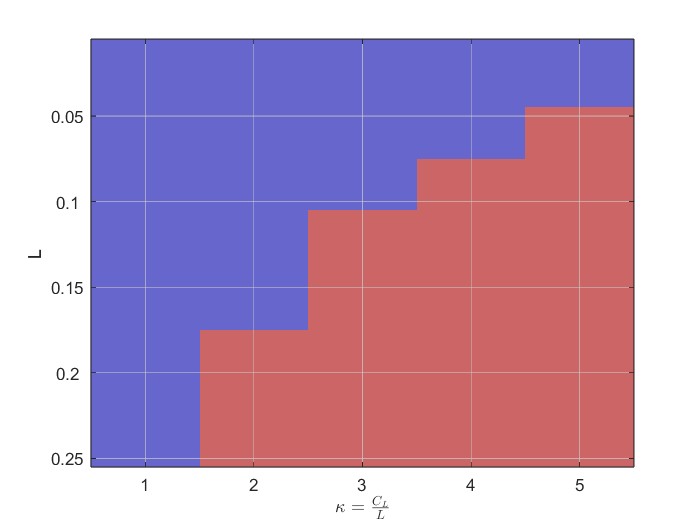}
			\caption*{{\blue Blue}: $r_{\text{MSGD}}> r_{\text{SGD}}$ \\
				{\red Red}: $r_{\text{MSGD}}\le  r_{\text{SGD}}$
			}
			\label{fig1:graph2}
		\end{subfigure}
		\hfill
		\begin{subfigure}{0.32\textwidth}
			\centering
			\includegraphics[width=\linewidth]{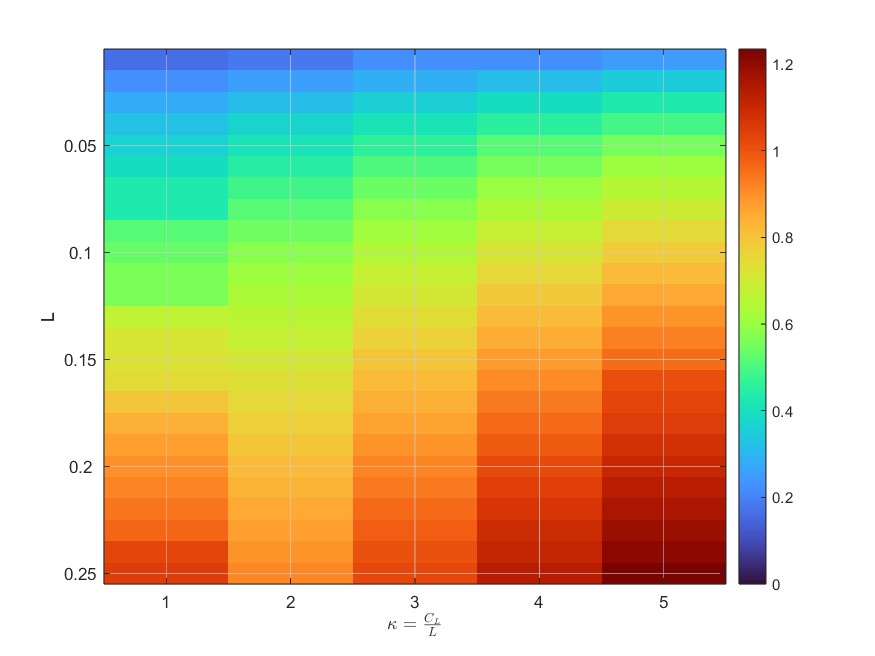}
			\caption*{Optimal friction $\mu^*$ \\ \ }
			\label{fig1:graph3}
		\end{subfigure}
		
		\caption{Comparison of the convergence rate $r_{\text{MSGD}}$ for MSGD and the convergence rate $r_{\text{SGD}}$ for SGD in the sense of \eqref{eq:rate} for fixed $\gamma=0.01$ and $\sigma=0$, different values of $L$ ($y$-axis) and $\kappa = \frac{C_L}{L}$ ($x$-axis) and optimally chosen friction parameter $\mu^*$. {\blue Blue} represents an outperformance of MSGD, {\red red} represents an outperformance of SGD.}
		\label{fig1}
	\end{figure}
	
	\begin{figure}[H] 
		\centering
		\begin{subfigure}{0.32\textwidth}
			\centering
			\includegraphics[width=\linewidth]{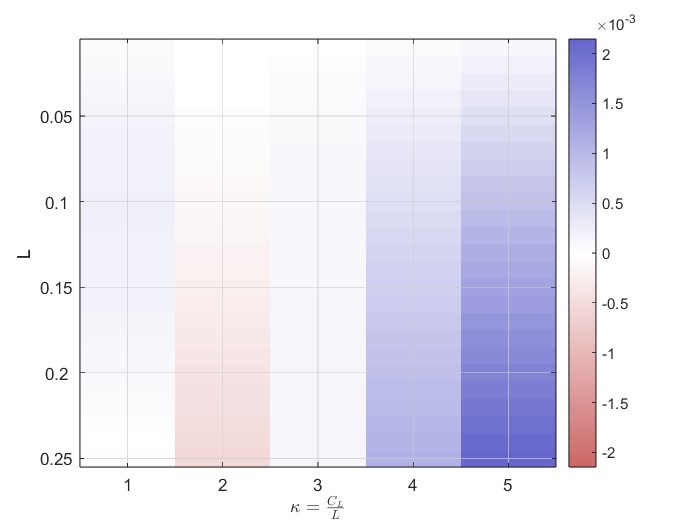}
			\caption*{$r_{\text{MSGD}}-r_{\text{SGD}}$\\ \ }
			\label{fig2:graph1}
		\end{subfigure}
		\hfill
		\begin{subfigure}{0.32\textwidth}
			\includegraphics[width=\linewidth]{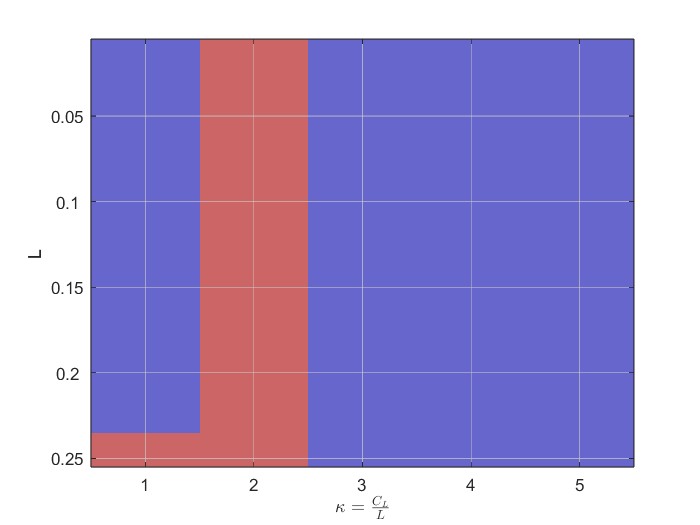}
			\caption*{{\blue Blue}: $r_{\text{MSGD}}> r_{\text{SGD}}$ \\
				{\red Red}: $r_{\text{MSGD}}\le  r_{\text{SGD}}$
			}
			\label{fig2:graph2}
		\end{subfigure}
		\hfill
		\begin{subfigure}{0.32\textwidth}
			\centering
			\includegraphics[width=\linewidth]{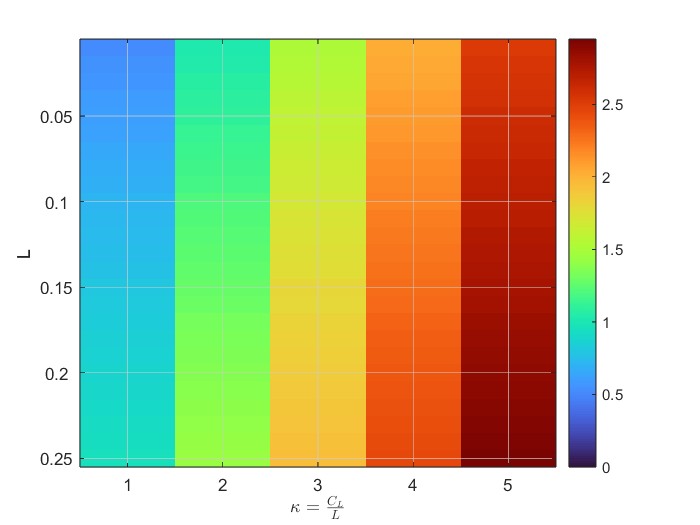}
			\caption*{Optimal friction $\mu^*$ \\ \ }
			\label{fig2:graph3}
		\end{subfigure}	
		\caption{Comparison of the convergence rate $r_{\text{MSGD}}$ for MSGD and the convergence rate $r_{\text{SGD}}$ for SGD in the sense of \eqref{eq:rate} for fixed $\gamma=0.01$ and $\sigma=100$, different values of $L$ ($y$-axis) and $\kappa = \frac{C_L}{L}$ ($x$-axis) and optimally chosen friction parameter $\mu^*$. 
			For $\kappa \ge 4$ one has $r_{\text{SGD}}<1$ so that SGD does not converge.}
		\label{fig2}
	\end{figure}
	
	\begin{figure}[t] 
		\centering
		\begin{subfigure}{0.32\textwidth}
			\centering
			\includegraphics[width=\linewidth]{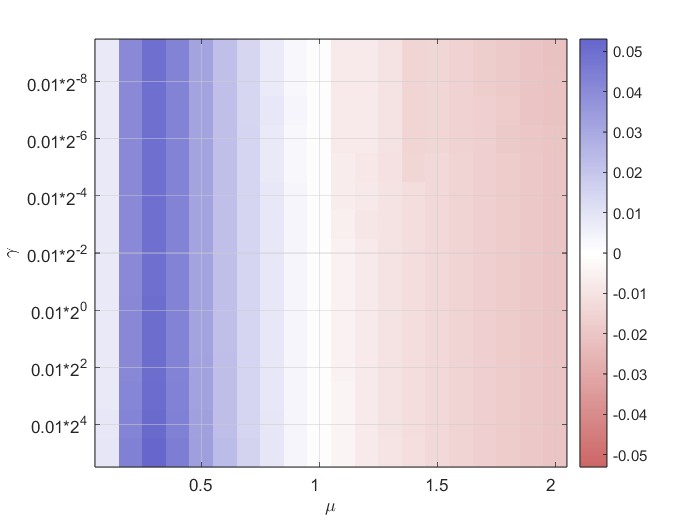}
			\caption*{$\frac{r_{\text{MSGD}}-r_{\text{SGD}}}{\gamma}$ for $\sigma=0$}
			\label{fig3:graph1}
		\end{subfigure}
		\hspace{2cm}
		\begin{subfigure}{0.32\textwidth}
			\centering
			\includegraphics[width=\linewidth]{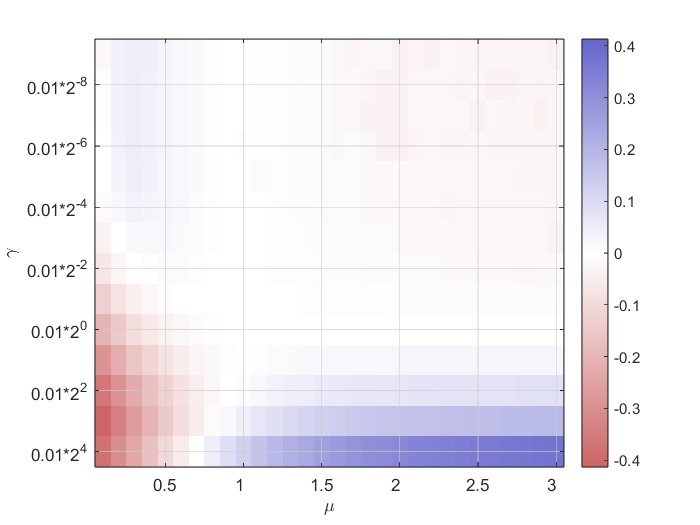}
			\caption*{$\frac{r_{\text{MSGD}}-r_{\text{SGD}}}{\gamma}$ for $\sigma=100$}
			\label{fig3:graph2}
		\end{subfigure}
		\caption{Comparison of the convergence rate $r_{\text{MSGD}}$ for MSGD and the convergence rate $r_{\text{SGD}}$ for SGD in the sense of \eqref{eq:rate} for fixed $L=\frac 1 {50}$, $C_L=\frac{3}{50}$ and different values of $\gamma$ ($y$-axis) and $\mu$ ($x$-axis). The figure shows the value $(r_{\text{MSGD}}-r_{\text{SGD}})/\gamma$.}
		\label{fig3}
	\end{figure}
	
	Note that $r_{\operatorname{MSGD}}$ is a rigorous, theoretical upper bound on the convergence rate of $\E[\1_{\IA_{n-1}}f(X_n)]$. In order to derive $r_{\operatorname{MSGD}}$ one has to solve a constrained optimization task, see Theorem~\ref{rem:constrained1}.	This constrained optimization task is executed by the \emph{fmincon} function in \emph{Matlab} using the interior point method. Therefore, it may be the case that $r_{\operatorname{MSGD}}$ is underestimated in Figure~\ref{fig1} and Figure~\ref{fig2}.

	We also compare $r_{\text{MSGD}}$ with $r_{\text{SGD}}$ for fixed problem parameters $L=\frac{1}{50}$, $C_L=\frac{3}{50}$ and $\sigma=0$, respectively $\sigma = 100$, see Figure~\ref{fig3}. We observe that, in the small step-size regime and with no stochastic noise ($\sigma=0$), there is a large interval of friction parameters that lead to an outperformance of MSGD over SGD. For large noise intensity ($\sigma=100$), the outperformance of MSGD is most notable in the mid step-size regime. 
	
	In particular, large step-sizes lead to large noise intensity in the corresponding continuous-in-time model which is shown to outperform continuous-in-time SGD in this scenario, see Remark~\ref{rem:compSDE}. However, this heuristic comparison is only feasible for sufficiently small step-sizes.
\end{remark}

\begin{remark}  \label{rem:Lip1}
	We discuss how one can weaken the global Lipschitz assumption on $\grad f$ if the stochastic noise is almost surely bounded. 
	Let $(\IA_n)_{n \in \N_0}$ be given by $\IA_n=\{X_i \in B_r(y) \text{ for all } i=0, \dots, n\}$ for a global minimum $y \in \R^d$ and assume that there exists a $\bar \gamma >0$ with $\sup_{n \in \N}\gamma_n \le \bar \gamma$ and $\bar \gamma \mu \le 1$. Moreover, assume that
	there exist deterministic constants $C_f, C_D\ge 0$ such that $|\grad f(x)|\le C_f$ for all $x \in B_r(y)$ and,
	for all $n \in \N_0$, we have $|D_{n+1}|\le C_D$ almost surely on $\{X_n \in B_r(y)\}$.
	Set $C_V=\frac{C_f+C_D}{\mu}$ and assume that $|V_0|\le C_V$ almost surely.
	Then, in all of the above statements it suffices to assume that $\grad f$ is $C_L$-Lipschitz continuous on $B_{(r+\bar \gamma C_V) \vee 2r}(y)$  and $0\le f(x)$ for all $x \in B_{r+\bar \gamma C_V}(y)$.
	
	Indeed, a simple induction argument shows that, for all $n \in \N_0$, $|V_{n+1}|\le C_V$ and, thus, $X_{n+1} \in B_{r+\bar \gamma C_V}(y)$ almost surely on the event $\IA_{n-1}$. Now, all Taylor estimates, see e.g. \eqref{eq:273856230}, hold under the assumption that $\grad f$ is $C_L$-Lipschitz continuous on $B_{r+\bar \gamma C_V}(y)$. Moreover, the Lipschitz continuity of $\grad f$ on $B_{2r}(y)$ implies the inverse PL-inequality on $B_r(y)$, see Lemma~\ref{rem:Lipschitz}.
\end{remark}

\subsection{Lyapunov estimates}
Let $a,b >0$ and let $(E_n)_{n \in \N_0}$ be the $(\cF_n)_{n \in \N_0}$-adapted stochastic process given by
\begin{align} \label{eq:Lyapunov}
	E_n = a f(X_{n}) + \langle \grad f (X_{n}), V_n \rangle +\frac{b}{2} |V_n|^2.
\end{align}
In our setting, $(E_n)_{n \in \N_0}$ plays the role of a random Lyapunov function. Although, in general, $(E_n)_{n \in \N_0}$ might take negative values, assuming the inverse PL-condition there exist choices for $a$ and $b$ such that $(E_n)_{n \in \N_0}$ is a non-negative process.

\begin{lemma} \label{lem:Epositive}
	Let $a, b, C_L>0$ and
	\begin{align*}
		E(x,y) = a f(x) + \langle \grad f(x), y \rangle + \frac{b}{2} |y|^2.
	\end{align*}
	If $ab\ge C_L$ then for all $x \in \R^d$ satisfying $2C_L f(x)\ge |\grad f(x)|^2$ we have $E(x,y)\ge 0$, for all $y \in \R^d$.
\end{lemma}

\begin{proof}
	If $\grad f(x)=0$ the statement is trivial. If $\grad f(x) \neq 0$, we denote $\rho=\frac{|y|}{|\grad f(x)|}$ and get
	$$
	E(x,y) \ge \Bigl(\frac{a}{2C_L}-\rho+\frac{b}{2}\rho^2 \Bigr) |\grad f(x)|^2.
	$$
	The quadratic function $\varphi(\rho)=\frac{a}{2C_L}-\rho+\frac{b}{2}\rho^2$ attains its global minimum at $\rho=\frac{1}{b}$ and using $ab \ge C_L$ we deduce that
	\begin{align*}
		\varphi\Bigl(\frac 1b\Bigr)=\frac{a}{2C_L}-\frac{1}{2b}\ge 0.
	\end{align*}
\end{proof}

In the next proposition, we derive a convergence statement for the MSGD scheme using the Lyapunov process $(E_n)_{n \in \N_0}$.

\begin{proposition} \label{prop:SHB}
	Let $L,a,b>0, \sigma \ge 0$. Let $(\IA_n)_{n \in \N_0}$ be a decreasing sequence of events such that, for all $n \in \N_0$, $\IA_n \in \cF_n$ and on $\IA_n$ it holds that
	\begin{align} \label{eq:1293182}
		|\grad f(X_n)|^2 \ge 2L f(X_n) \quad \text{and} \quad \E[|D_{n+1}|^2 | \cF_n] \le \sigma f(X_n).
	\end{align}
	Let $(\alpha_n)_{n \in \N}, (\beta_n)_{n \in \N}, (\delta_n)_{n \in \N}$ and $(\epsilon_n)_{n \in \N}$ be given by $\eqref{eq:alpha}$. Assume that $(\beta_n)_{n \in \N}$, $(\alpha_{n+1}-a\delta_{n+1}+2\beta_{n+1}L)_{n \in \N_0}$, $(\epsilon_{n+1}-\frac{b}{2} \delta_{n+1})_{n \in \N_0}$ and $(\frac{C_L}{2} \gamma_{n}^2- \frac{b}{2}\gamma_{n})_{n \in \N_0}$ are sequences of non-positive reals and $(\delta_n)_{n \in \N}$ is a sequence of non-negative reals.
	Moreover, if $\P(\bigcap_{n \in \N_0} \IA_n)<\P(\IA_0)$ additionally assume that $ab \ge C_L$.
	Then, for all $n \in \N$ it holds that
	\begin{align}\begin{split} \label{eq:238477}
			& \E[\1_{\IA_{n-1}} f(X_{n})]\le \Bigl(\prod_{i=1}^{n} (1+a\gamma_{i})^{-1}\Bigr) \\
			& \qquad \qquad \qquad \Bigl( \E[\1_{\IA_0} f(X_0)] + \sum_{i=1}^{n}  \frac{\gamma_{i}}{1+a\gamma_{i}} \Bigl(\prod_{j=1}^{i} (1+a\gamma_{j})\Bigr) \Bigl(\prod_{j=1}^{i}\delta_{j} \Bigr)\E[\1_{\IA_0}E_0] \Bigr).
		\end{split} 
	\end{align}	
\end{proposition}

\begin{proof}
	In a first step, we derive a convergence rate for the expectation of the Lyapunov process $(E_n)_{n \in \N_0}$. For this, we consider the time evolution of the three summands in \eqref{eq:Lyapunov}, separately.
	
	First, we look at the evolution of $(f(X_{n}))_{n \in \N_0}$.
	Let $x,y \in \R^d$ and note that with the Lipschitz-continuity of $\grad f$ we get
	\begin{align}
		\begin{split}
			\label{eq:Taylorp1}
			f(y)&\le f(x) + \langle \grad f(x),y-x\rangle  + \frac{C_L}{2} |y-x|^2.
		\end{split}
	\end{align}
	
	Now, for $n \in \N_0$ we use \eqref{eq:Taylorp1} with $x=X_{n}$ and $y=X_{n+1}$ and \eqref{eq:MSGDintro} to get 
	\begin{align} \begin{split} 
			\label{eq:273856230}
			&\E[\1_{\IA_n}f(X_{n+1})]\le \E\Big[\1_{\IA_n}\bigl(f(X_{n})-\gamma_{n+1}^2 |\grad f(X_n)|^2 \\
			&\qquad \qquad \qquad \qquad \quad + (\gamma_{n+1}-\gamma_{n+1}^2\mu) \langle \grad f(X_{n}),V_n\rangle + \frac{C_L}{2} \gamma_{n+1}^2 \bigl|V_{n+1}\bigr|^2 \bigr) \Big].
		\end{split}
	\end{align}
	
	Next, we control the evolution of $(|V_n|^2)_{n \in \N_0}$. Using \eqref{eq:1293182}, we get
	\begin{align} \begin{split}
			\E[\1_{\IA_n}|V_{n+1}|^2]
			\le  \E[\1_{\IA_n}(&(1-2\gamma_{n+1}\mu+\gamma_{n+1}^2\mu^2)|V_n|^2+ \gamma_{n+1}^2 |\grad f(X_n)|^2  \\
			&  - 2 (\gamma_{n+1}-\gamma_{n+1}^2\mu) \langle V_n, \grad f(X_n)\rangle  + \gamma_{n+1}^2 \sigma f(X_n))].
		\end{split}
	\end{align}
	Lastly,
	\begin{align} \begin{split} \label{eq:1764443}
			\E[\1_{\IA_n} &\langle \grad f(X_{n+1}), V_{n+1} \rangle] = \E[ \1_{\IA_n} (\langle \grad f(X_{n}), V_{n+1} \rangle + \langle \grad f(X_{n+1})- \grad f(X_{n}), V_{n+1} \rangle )] \\
			&\le \E[\1_{\IA_n} ((1-\gamma_{n+1}\mu) \langle \grad f(X_n), V_n \rangle - \gamma_{n+1}  |\grad f(X_n)|^2 + C_L \gamma_{n+1} |V_{n+1}|^2)].
		\end{split}
	\end{align}
	
	Combining the estimates \eqref{eq:273856230}-\eqref{eq:1764443}, we obtain
	\begin{align*}
		&\E[\1_{\IA_{n}} E_{n+1}] \\
		&\le \E[\1_{\IA_{n}}(\alpha_{n+1} f(X_n) + \beta_{n+1} |\grad f(X_n)|^2 +\delta_{n+1} \langle \grad f(X_n),V_n \rangle + \epsilon_{n+1} |V_n|^2)],
	\end{align*}
	where
	\begin{align} \label{eq:alpha}
		\begin{split}
			\alpha_{n+1} = &a + \Bigl(\frac{b}{2}+C_L \gamma_{n+1}\Bigl(1+\frac{\gamma_{n+1}a}{2}\Bigr)\Bigr)\gamma_{n+1}^2 \sigma, \\
			\beta_{n+1} = &- \gamma_{n+1} - a\gamma_{n+1}^2 + \Bigl(\frac{b}{2}+C_L \gamma_{n+1}\Bigl(1+\frac{\gamma_{n+1}a}{2}\Bigr)\Bigr)\gamma_{n+1}^2,\\
			\delta_{n+1} = &1-\mu\gamma_{n+1}+ a(\gamma_{n+1}-\gamma_{n+1}^2\mu) \\
			&- 2\Bigl(\frac{b}{2}+C_L \gamma_{n+1}\Bigl(1+\frac{\gamma_{n+1}a}{2}\Bigr)\Bigr) (\gamma_{n+1}-\gamma_{n+1}^2\mu), \\
			\epsilon_{n+1} = &\Bigl(\frac{b}{2}+C_L \gamma_{n+1}\Bigl(1+\frac{\gamma_{n+1}a}{2}\Bigr)\Bigr) (1-2\gamma_{n+1}\mu+\gamma_{n+1}^2\mu^2).
		\end{split}
	\end{align}
	By definition
	$
	\E[\1_{\IA_n} \langle \grad f(X_n), V_n \rangle ] =  \E[\1_{\IA_n}(E_n -a f(X_n)-\sfrac{b}{2} |V_n|^2)],
	$
	so that, using the PL-inequality \eqref{eq:1293182} and the fact that $(\beta_n)_{n \in \N}$ is non-positive, we get 
	\begin{align} \begin{split} \label{eq:579999663}
			&\E[\1_{\IA_{n}}E_{n+1}] \\
			&\le \E[\1_{\IA_n}(\delta_{n+1}E_n + (\alpha_{n+1}-a\delta_{n+1}+2\beta_{n+1}L) f(X_n)  + (\epsilon_{n+1}-\frac{b}{2} \delta_{n+1}) |V_n|^2)].
		\end{split}
	\end{align}
	With the assumptions on $(\alpha_{n+1}-a\delta_{n+1}+2\beta_{n+1}L)_{n \in \N_0}$ and $(\epsilon_{n+1}-\frac{b}{2} \delta_{n+1})_{n \in \N_0}$ we have
	$
	\E[\1_{\IA_{n}} E_{n+1}] \le \delta_{n+1} \E[\1_{\IA_{n}} E_n].
	$
	For $n\in \N$, we use Lemma~\ref{lem:Epositive} and the monotonicity of $(\IA_n)_{n \in \N_0}$ in order to show that $\E[\1_{\IA_{n}} E_n]\le \E[\1_{\IA_{n-1}} E_n]$ so that, iteratively, 
	\begin{align} \label{eq:2776362562}
		\E[\1_{\IA_{n-1}} E_{n}] \le \Bigl( \prod_{i=1}^{n}\delta_{i} \Bigr)\E[\1_{\IA_{0}} E_0].
	\end{align}
	Next, we bound the expectation of $(f(X_n))_{n \in \N_0}$ using \eqref{eq:2776362562}. Analogously to \eqref{eq:273856230}, we have for all $n \in \N_0$ that
	\begin{align*}
		\E[\1_{\IA_{n}}f(X_{n+1})]\le \E\Big[\1_{\IA_{n}} \bigl( f(X_{n})+\gamma_{n+1}\langle \grad f(X_{n+1}),V_{n+1}\rangle + \frac{C_L}{2} \gamma_{n+1}^2 \bigl|V_{n+1}\bigr|^2\bigr) \Big]
	\end{align*}
	so that, by definition of $E_{n+1}$,
	\begin{align*} 
		&(1+a\gamma_{n+1})\E[\1_{\IA_{n}}f(X_{n+1})]\\
		&\le \E\Big[ \1_{\IA_{n}} \bigl(f(X_{n})+\gamma_{n+1} E_{n+1} + (\frac{C_L}{2} \gamma_{n+1}^2- \frac{b}{2}\gamma_{n+1}) \bigl|V_{n+1}\bigr|^2 \bigr)\Big].
	\end{align*}
	By assumption, $(\frac{C_L}{2} \gamma_{n}^2- \frac{b}{2}\gamma_{n})_{n \in \N}$ is a sequence of non-positive reals. Therefore, we can neglect the last term in the upper bound above and get
	\begin{align} \label{eq:21378861}
		\E[\1_{\IA_{n}}f(X_{n+1})] \le (1+a\gamma_{n+1})^{-1} \E[\1_{\IA_n} f(X_{n})]+\frac{\gamma_{n+1}}{1+a\gamma_{n+1}} \Bigl( \prod_{i=1}^{n+1}\delta_{i} \Bigr)\E[\1_{\IA_0}E_0].
	\end{align}
	Using the non-negativity of $f(X_n)$, the monotonicity of $(\IA_n)_{n \in \N_0}$ and \eqref{eq:21378861}, one can inductively show that, for all $n \in \N$,
	\begin{align*}
		&\E[\1_{\IA_{n-1}} f(X_{n})] \\
		&\le \Bigl(\prod_{i=1}^{n} (1+a\gamma_{i})^{-1}\Bigr) \Bigl( \E[\1_{\IA_0} f(X_0)] + \sum_{i=1}^{n}  \frac{\gamma_{i}}{1+a\gamma_{i}} \Bigl(\prod_{j=1}^{i} (1+a\gamma_{j})\Bigr) \Bigl(\prod_{j=1}^{i}\delta_{j} \Bigr)\E[\1_{\IA_0}E_0] \Bigr).
	\end{align*}
\end{proof}

\begin{proof} [Proof of Theorem~\ref{rem:constrained1}] 
	Applying Proposition~\ref{prop:SHB} in the case of a constant sequence of step-sizes $\gamma_n \equiv \gamma >0$, the assumptions on the parameters read exactly as in Theorem~\ref{rem:constrained1}.
	Now, for parameters $a,b,\mu, \gamma$ that satisfy all of the inequalities stated in Theorem~\ref{rem:constrained1} and under the remaining assumptions of Proposition~\ref{prop:SHB}, we get for all $n\in \N$ that
	\begin{align*}
		\E[\1_{\IA_{n-1}} f(X_{n})] \le (1+a\gamma)^{-n} \Bigl( \E[\1_{\IA_0} f(X_0)] + \frac{\gamma}{1+a\gamma} \E[\1_{\IA_0}E_0] \sum_{i=1}^{n}   ( (1+a\gamma) \delta)^i \Bigr).
	\end{align*}
	Thus, if $(1+a\gamma)\delta =1$ we get for a constant $C \ge 0$ that
	$$
	\E[\1_{\IA_{n-1}} f(X_{n})] \le C(1+a\gamma)^{-n}n.
	$$
	If $(1+a\gamma) \delta \neq 1$ we have
	$$
	\sum_{i=1}^n ((1+a\gamma)\delta)^i = \frac{1-\bigl((1+a\gamma)\delta \bigr)^{n+1}}{1-(1+a\gamma)\delta}.
	$$
\end{proof}

\subsection{The small step-size case}
In this section, we consider the situation of sufficiently small step-sizes $(\gamma_n)_{n \in \N}$ and prove the main results for the MSGD process, Theorem~\ref{theo1} and Theorem~\ref{theo2}.

\begin{proposition} \label{prop:convrate}
	Let $L, a,b,\mu > 0$ and $\sigma \ge 0$. Let $(\IA_n)_{n \in \N_0}$ be a decreasing sequence of events such that, for all $n \in \N_0$, $\IA_n \in \cF_n$ and on $\IA_n$ it holds that
	\begin{align*} 
		|\grad f(X_n)|^2 \ge 2L f(X_n)  \quad \text{and} \quad \E[|D_{n+1}|^2 | \cF_n] \le \sigma f(X_n).
	\end{align*} 
	Assume that
	\begin{align} \label{eq:238877}
		\mu-a+b>0 \quad , \quad a\mu-a^2+ab-2L <0 \quad \text{and} \quad C_L -\frac{b}{2}(\mu+a-b)<0.
	\end{align}
	If $\P(\bigcap_{n \in \N_0} \IA_n)< \P(\IA_0)$ we additionally assume that $ab \ge C_L$.
	Then, for every $0<\eps<m:= \min (a, \mu-a+b)$ there exist constants $C, \bar \gamma \ge 0$ such that if $\sup_{n \in \N}\gamma_n \le \bar \gamma$ it holds that 
	\begin{enumerate}
		\item[(i)] 
		$$
		\max (\E[\1_{\IA_{n-1}} f(X_n)], 	\E[\1_{\IA_{n-1}}|V_n|^2] ) \le C \exp(- (m-\eps) t_n)
		$$
		for all $n \in \N$, where $t_n = \sum_{i=1}^n \gamma_i$.
		\item[(ii)] Let $m'<m-\eps$ and assume that $ \sum_{i=0}^\infty \exp((m'-(m-\eps)) t_i)<\infty$. Then $\exp(m' t_n) f(X_n) \to 0$ almost surely on the event $\IA_\infty= \bigcap_{n \in \N_0} \IA_n$.
		\item[(iii)] The process $(X_n)_{n \in \N_0}$ converges almost surely on $\IA_\infty$.
	\end{enumerate}
\end{proposition}

\begin{proof}
	(i): 
	First, note that, for all $x\in \R$, $1+x\le \exp(x)$ and $(1+x)^{-1}=e^{-x+o(x)}$. Thus, for every $\eps'>0$ there exists a $\bar \gamma'>0$ such that if $\max_{i=1, \dots, n} \gamma_i \le \bar \gamma'$ we have
	$$
	\prod_{i=1}^{n} (1+a\gamma_{i})^{-1} \le \Bigl(\prod_{i=1}^{n} \exp((-a+\eps')\gamma_i)\Bigr) = \exp((-a+\eps')t_n ).
	$$
	Moreover, for $(\delta_n)_{n \in \N}$ given in \eqref{eq:alpha} we have $\delta_n = 1-\gamma_n(\mu-a+b)+o(\gamma_n)$ so that, for all $\eps''>0$, there exists a $\bar \gamma ''>0$ such that if $\max_{i=1, \dots, n} \gamma_i \le \bar \gamma''$ we have $\delta_i\ge 0$ for all $i=1, \dots, n$ and
	$$
	\prod_{i=1}^{n}\delta_{i} \le \exp(-(\mu-a+b-\eps'')t_n).
	$$
	Note that $\beta_n = -\gamma_n + o(\gamma_n)$, $\alpha_n-a\delta_n+2\beta_n L =  \gamma_n (a\mu-a^2+ab-2 L)+o(\gamma_n)$, 
	$$
	\epsilon_n-\frac b2 \delta_n =  \gamma_n \Bigl(C_L-\frac{b}{2}(\mu+a-b)\Bigr) +o(\gamma_n) \;  \text{ and } \; \sfrac {C_L}{2}\gamma^2-\sfrac{b}{2}\gamma=-\sfrac{b}{2}\gamma+o(\gamma),
	$$
	and using assumption \eqref{eq:238877} we can choose $\bar \gamma \le \min(\bar \gamma',\bar \gamma'')$ sufficiently small such that if $\sup_{n \in \N}\gamma_n \le \bar \gamma$ all of the above terms are strictly negative. Then, using Proposition~\ref{prop:SHB} we get that for all $n \in \N$
	\begin{align*}
		\E[\1_{\IA_{n-1}} f(X_{n})] 
		\le \exp((&-a+\eps'+\eps'')t_n )  \\
		& \Bigl( \E[ \1_{\IA_0}f(X_0)]+ \sum_{i=1}^{n}  \gamma_{i}  \exp(-(\mu-2a+b)t_i) \E[\1_{\IA_0}E_0] \Bigr).
	\end{align*}
	If $a< \mu-a+b$, the function $t \mapsto \exp(-(\mu-2a+b)t)$ is monotonously decreasing and we get
	\begin{align*}
		\sum_{i=1}^{n}  \gamma_{i}  \exp(-(\mu-2a+b)t_i) &\le \int_{0}^{t_n} \exp(-(\mu-2a+b)t) \, dt.
	\end{align*}
	Thus,
	$$
	\E[\1_{\IA_{n-1}} f(X_{n})] \le\exp((-a+\eps'+\eps'')t_n ) \Bigl( \E[\1_{\IA_0} f(X_0)] + \frac{\E[\1_{\IA_0}E_0]}{\mu-2a+b}  \Bigr).
	$$
	For $a>\mu-a+b$,  the function $t \mapsto \exp(-(\mu-2a+b)t)$ is monotonously increasing and we get
	\begin{align*}
		\sum_{i=1}^{n}  &\gamma_{i}  \exp(-(\mu-2a+b)t_i) \le \int_{0}^{t_n} \exp(-(\mu-2a+b)(t+\bar \gamma)) \, dt.
	\end{align*}
	Thus,
	\begin{align*}
		\E[\1_{\IA_{n-1}} f(X_{n})] \le&\exp((-(\mu-a+b)+\eps'+\eps'')t_n ) \Bigl( \E[\1_{\IA_0} f(X_0)] \\
		& \qquad \qquad \qquad  +\frac{\E[\1_{\IA_0}E_0]}{2a-\mu-b} \bigl(\exp(-(\mu-2a+b)\bar \gamma)  \Bigr).
	\end{align*}
	Lastly, for $a=\mu-a+b$ we get 
	$
	\sum_{i=1}^{n}  \gamma_{i}  \exp(-(\mu-2a+b)t_i)=t_n
	$
	and, thus,
	$$
	\E[\1_{\IA_{n-1}}f(X_n)] \le \exp((-a+\eps'+\eps'')t_n ) \Bigl( \E[ \1_{\IA_0}f(X_0)] + t_n\E[\1_{\IA_0}E_0] \Bigr).
	$$
	Note that $\exp(-\eps''' t)t \to 0$ for all $\eps'''>0$. Therefore, there exists a constant $C'>0$ such that
	$$
	\E[\1_{\IA_{n-1}}f(X_{n})] \le \exp((-a+\eps'+\eps''+\eps''')t_n ) \Bigl( \E[\1_{\IA_0} f(X_0)] + C'\E[\1_{\IA_0}E_0] \Bigr).
	$$
	The proof of the first assertion follows by choosing $\eps'=\eps''= \eps''' =\frac 13 \eps$.
	
	For the second assertion, note that, by Lemma~\ref{rem:Lipschitz}, we have $|\grad f(x)|^2 \le 2 C_L f(x)$ for all $x \in \R^d$
	and, using that $f(X_n)\ge 0$, we get by the Cauchy-Schwarz inequality
	\begin{align*}
		\frac b2 |V_n|^2 &\le E_n - \langle \grad f(X_n), V_n \rangle \le E_n + |\grad f(X_n)| \, |V_n|,
	\end{align*}
	where $E_n$ is defined by \eqref{eq:Lyapunov}.
	Thus, using Young's inequality, 
	\begin{align} \label{eq:Yconv2}
		\frac b2 \E\bigl[\1_{\IA_{n-1}}  |V_n|^2\bigr]
		&\le \E[\1_{\IA_{n-1}} E_n] + \frac 1b \E[\1_{\IA_{n-1}} |\grad f(X_n)|^2] + \frac b4\E[ \1_{\IA_{n-1}} |V_n|^2] 
	\end{align}
	and, with the bound for $\E[\1_{\IA_{n-1}} f(X_n)]$ and \eqref{eq:2776362562}, we get a constant $C\ge 0$ such that 
	\begin{align} \label{eq:Yconv} 
		\E[\1_{\IA_{n-1}}|V_n|^2]^{1/2} \le C \exp\Bigl(-\frac 12 (m -\eps)t_n\Bigr).
	\end{align}
	
	(ii): Let $m'<m-\eps$. For $\eps' >0$ and $n \in \N$ consider the set $\IB_n = \IA_\infty \cap \{\sup_{i \ge n} \exp(m' t_i) f(X_i) \ge \eps' \}$.
	With the Markov inequality and (i) there exists a $C>0$ such that
	\begin{align*}
		\P(\IB_n) &\le \sum_{i=n}^\infty \P\bigl(\IA_{i-1} \cap \bigl\{\exp(m' t_i) f(X_i) \ge \frac{\eps'}{2}\bigr\}\bigr) \\
		& \le \sum_{i=n}^\infty \frac{2\exp(m' t_i)}{\eps'} \E[\1_{\IA_{i-1}} f(X_i)] \le \frac{C}{\eps'} \sum_{i=n}^\infty \exp((m'-(m-\eps)) t_i) \overset{n \to \infty}{\longrightarrow} 0.
	\end{align*}
	With
	$$
	\P\Bigl(\IA_\infty \cap \big\{\limsup_{n \to \infty} \exp(m' t_n)f(X_n) \ge \eps'\big\}\Bigr) \le \P\Bigl(\bigcap_{n \in \N} B_n\Bigr) =0
	$$
	we get $\exp(m' t_n)f(X_n)\to 0$ almost surely on $\IA_\infty$.
	
	(iii):
	We consider the event $\IA_\infty$ and bound the distance that the process $(X_n)_{n \in \N_0}$ travels. Since $\eps < m$ the mapping $t  \mapsto \exp(-\frac 12 (m -\eps)t)$ is monotonously decreasing. Thus, using \eqref{eq:Yconv} we get
	\begin{align*}
		\E\Bigl[\1_{\IA_\infty} \sum_{i=1}^\infty |X_i-X_{i-1}|\Bigr] &\le   \sum_{i=1}^\infty \gamma_i \E[\1_{\IA_{i-1}}  |V_i|^2]^{1/2} \le C \sum_{i=1}^\infty \gamma_i  \exp\Bigl(-\frac 12 (m -\eps)t_i\Bigr)\\
		&\le C \int_0^\infty \exp\Bigl(-\frac 12 (m -\eps)t\Bigr) \, dt < \infty,
	\end{align*}
	which implies that $\sum_{i=1}^\infty |X_i-X_{i-1}|$ is almost surely finite, on $\IA_\infty$. Thus, $(X_n)_{n \in \N_0}$ almost surely converges on $\IA_{\infty}$.
\end{proof}

\begin{lemma} \label{lem:feasible1}
	For all $\mu >0$ there exist $a,b >0$ such that $ab \ge C_L$ and \eqref{eq:238877} is satisfied.
\end{lemma}

\begin{proof}
	Let $\eps,b>0$ and choose $a=b+2\frac{C_L}{b}-\mu+\eps$. Note that $a>0$ iff $\mu < b+2\frac {C_L}{b}+\eps$. Moreover, $\mu-a+b=2(\mu-\frac{C_L}{b})-\eps$ is positive iff $\mu > \frac{C_L}{b}+\frac{\eps}{2}$. Now, 
	$
	C_L-\frac b2(\mu+a-b)=-\frac{b}{2}\eps < 0
	$
	and we have 
	$$
	a\mu-a^2+ab-2L = -2\mu^2+2\mu\Bigl(b+\frac{3C_L}{b}+\frac{3\eps}{2}\Bigr)-2C_L-2L-4\frac{C_L^2}{b^2}-2\eps\Bigr(\frac b2 + 2\frac{C_L}{b}+\frac{\eps}{2}\Bigl).
	$$
	The latter term is a quadratic function in $\mu$ that is negative outside of the two roots. Therefore, $a\mu-a^2+ab-2L <0$ iff $\mu \notin [\mu_-^{\eps,b},\mu_+^{\eps,b}]$, where
	\begin{align} \label{eq:mu}
		\mu_{\pm}^{\eps,b} = \frac 12 \Bigl( b+\frac{3C_L}{b}+\frac{3\eps}{2} \pm \sqrt{\Bigl(b+\frac{C_L}{b}+\frac{\eps}{2}\Bigr)^2-4L} \,  \Bigr).
	\end{align}
	Note that $(b+\frac{C_L}{b})^2\ge 4C_L$, for all $b >0$, and $C_L\ge L$ so that $\mu_{\pm}^{\eps,b}$ is well-defined.
	Moreover, $\mu_-^{\eps,b}>\frac{C_L}{b}+\frac{\eps}{2}$ and $\mu_+^{\eps,b}<b+2\frac{C_L}{b}+\eps$. The additional assumption $ab\ge C_L$ is satisfied iff $\mu\le b+ \frac{C_L}{b}+\eps$. Therefore, the set of friction parameters $\mu$ that satisfy \eqref{eq:238877} for the given pair $(a,b)$ is equal to
	$
	(\frac{C_L}{b}+\frac{\eps}{2},\mu_-^{\eps,b}) \cup (\mu_+^{\eps,b}, b+2\frac{C_L}{b}+\eps)
	$
	and the set of friction parameters $\mu$ that satisfy both \eqref{eq:238877} and $ab \ge C_L$ for the given pair $(a,b)$ contains the interval
	$
	(\sfrac{C_L}{b}+\sfrac{\eps}{2}, \mu_-^{\eps,b} \wedge (b+\sfrac{C_L}{b}+\eps)).
	$
	For all $b>0$, the latter interval is non-empty, the upper and lower limits are continuous in $b$ and the lower limit satisfies 
	$$
	\frac{C_L}{b}+\frac{\eps}{2} \overset{b \to \infty}{\longrightarrow} \frac{\eps}{2} \quad \text{ and } \quad \frac{C_L}{b}+\frac{\eps}{2} \overset{b \to 0}{\longrightarrow} \infty.
	$$
	By letting $\eps \to 0$ we showed that for every $\mu>0$ there exists a pair $(a,b)$ such that \eqref{eq:238877} is satisfied and $ab \ge C_L$.
\end{proof}

\begin{proof} [Proof of Theorem~\ref{theo1}]
	By Lemma~\ref{lem:feasible1}, there exist parameters $a,b>0$ such that $ab\ge C_L$ and \eqref{eq:238877} is satisfied. Note that the choice $(\IA_n)_{n \in \N} = (\{X_i \in \mathcal D \text{ for all } i =0, \dots, n\})_{n \in \N}$ satisfies the assumptions of Proposition~\ref{prop:convrate}.
	Now, statements (i), (ii) and (iii) follow from Proposition~\ref{prop:convrate}.
\end{proof}

We give a general statement on the size of the convergence rate that still depends on the technical parameter $b$.

\begin{lemma} \label{lem:49378}
	Let $L, b, \eps > 0$ and $\sigma \ge 0$. Let $(\IA_n)_{n \in \N_0}$ be a decreasing sequence of events such that, for all $n \in \N_0$, $\IA_n \in \cF_n$ and on $\IA_n$ it holds that
	\begin{align*} 
		|\grad f(X_n)|^2 \ge 2L f(X_n) \quad \text{and} \quad \E[|D_{n+1}|^2 | \cF_n] \le \sigma f(X_n).
	\end{align*} 
	Let 
	$
	\mu \in (\frac{C_L}{b}+\frac{\eps}{2}, \mu_-^{\eps,b}) \cup (\mu_+^{\eps,b}, b+2\frac{C_L}{b}+\eps),
	$
	where $\mu_{\pm}^{\eps,b}$ is defined by \eqref{eq:mu}.
	If $\P(\bigcap_{n \in \N_0}\IA_n)<\P(\IA_0)$ additionally assume that $\mu \le b+\frac{C_L}{b}+\eps$. Then, there exist $C, \bar \gamma>0$ such that if $\sup_{n \in \N}\gamma_n \le \bar \gamma$ we have
	$$
	\E[\1_{\IA_{n-1}} f(X_n)] \le C \exp(-m(\eps,b,\mu) t_n),
	$$
	where
	\begin{align*} 
		m(\eps,b,\mu) = \begin{cases} 
			2(\mu-\frac{C_L}{b}-\eps), & \text{ if } \mu < \frac{1}{3} (b+4\frac{C_L}{b}+2\eps) \\
			b+2\frac{C_L}{b}-\mu, & \text{ if } \mu \ge \frac{1}{3} (b+4\frac{C_L}{b}+2\eps).
		\end{cases}
	\end{align*}
\end{lemma}

\begin{proof}
	Let $\eps, b>0$ and $\mu$ as in the assumptions of the lemma and set $a=b+2\frac{C_L}{b}-\mu+\eps$. Recall that in the proof of Lemma~\ref{lem:feasible1} we showed that this choice of parameters satisfies \eqref{eq:238877} and if $\P(\bigcap_{n \in \N_0}\IA_n)<\P(\IA_0)$, additionally, $ab \ge C_L$. Hence, we can apply Proposition~\ref{prop:convrate} and deduce that there exist constants $C, \bar \gamma >0$ such that, if $\sup_{n \in \N}\gamma_n \le \bar \gamma$, we have
	$$
	\E[\1_{\IA_{n-1}}f(X_n)] \le C \exp(-m(\eps,b,r) t_n),
	$$
	where
	$$
	m(\eps,b,\mu)= \min(a,\mu-a+b)-\eps= \min\Bigl(b+2\frac{C_L}{b}-\mu,2\Bigl(\mu-\frac{C_L}{b}-\eps\Bigr)\Bigr).
	$$
	The evaluation of the minimum is straight-forward.
\end{proof}

\begin{proof} [Proof of Theorem~\ref{theo2}]
	We maximize the convergence rate derived in Lemma~\ref{lem:49378} over all admissible parameters $\mu$ and $b$.
	First, assume that $\kappa=\frac{C_L}{L}<\frac 98$. Then, we have for all sufficiently small $\eps>0$ that $\kappa^\eps := \bigl(\frac{\sqrt {C_L}}{\sqrt{L}}+\frac{\eps}{4\sqrt{L}}\bigr)^2<\frac{9}{8}$ so that Set \begin{align*}
		b_{\pm}^\eps := \frac{3}{\sqrt 8} \sqrt L -\frac{\eps}{4} \pm \sqrt{\Bigl( \frac{3}{\sqrt 8} \sqrt L -\frac{\eps}{4} \Bigr)^2-C_L}
	\end{align*}
	is well-defined.
	For $b \in (b_-^\eps, b_+^\eps)$ we have $\frac{1}{3} (b+4\frac{C_L}{b}+2\eps)<\mu_-^{\eps, b}$, such that with Lemma~\ref{lem:49378} we get for $\mu \in (\frac{C_L}{b}+\frac{\eps}{2},\frac 13(b+4\frac{C_L}{b}+2\eps))$ the convergence rate $m(\eps, b ,\mu)=2\bigl( \mu-\frac{C_L}{b} -\eps\bigr)$. Note that
	$$
	\frac{d}{d\eps}\big|_{\eps=0} \frac 13 \Bigl(b_+^\eps-\eps^2+4\frac{C_L}{b_+^\eps-\eps^2}+2\eps\Bigr) =\frac 14 + \frac{3 \sqrt L}{4 \sqrt{9L-8C_L}}>0.
	$$ 
	Therefore, for sufficiently small $\eps$, we have 
	$$
	\mu^*_1 := \frac{1}{3} \Bigl(b_+^0+4\frac{C_L}{b_+^0}\Bigr) = \frac{1}{\sqrt 8}(5-\sqrt{ 9 - 8\kappa})\sqrt{L}<\frac 13 \Bigl(b_+^\eps-\eps^2+4\frac{C_L}{b_+^\eps-\eps^2}+2\eps\Bigr)
	$$
	and we
	get by continuity that
	\begin{align*}
		m(\eps, b_+^\eps-\eps^2,\mu^*_1) &\overset{\eps \to 0}{\longrightarrow} \frac 23 \Bigl(b_+^0+\frac{C_L}{b_+^0}\Bigr) = \frac 23 \frac{(b_+^0)^2+C_L}{b_+^0} = \sqrt {2L}.
	\end{align*}
	Moreover, note that $b_+^0$ satisfies $\frac 13 (b_+^0+4\frac{C_L}{b_+^0})<b_+^0+\frac{C_L}{b_+^0}$ since 
	$$
	\frac{C_L}{b_+^0} = \frac {1}{\sqrt 8} (3 \sqrt L -\sqrt{9L-8C_L})<\frac{1}{\sqrt 2}(3\sqrt L +\sqrt{9L-8 C_L})=2b_+^0.
	$$
	Thus, $b_+^0 (b_+^0+2\frac{C_L}{b_+^0}-\mu^*_1) > C_L$ and, for sufficiently small $\eps>0$, the parameters $b= b_+^\eps-\eps^2$ and $a=b_+^\eps-\eps^2+2\frac{C_L}{b_+^\eps-\eps^2}-\mu^*_1+\eps$ satisfy $ab\ge C_L$.
	
	Analogously, we get the convergence rate $m(\eps,b,\mu)=b+2\frac{C_L}{b}-\mu$ if $b \in (b_-^\eps,b_+^\eps)$ and $\mu \in (\frac 13(b+4\frac{C_L}{b}+2\eps),\mu_-^{\eps,b})$. Note that, since $\kappa < \frac 98$ we have
	$$
	\frac{d}{d\eps}\big|_{\eps=0} \frac 13 \Bigl(b_-^\eps+\eps^2+4\frac{C_L}{b_-^\eps+\eps^2}+2\eps\Bigr) = \frac 14 - \frac{3 \sqrt{L}}{4\sqrt{9L-8C_L}}<0.
	$$
	Therefore, for sufficiently small $\eps$, we have
	$$
	\mu^*_2 := \frac{1}{3} \Bigl(b_-^0+4\frac{C_L}{b_-^0}\Bigr) = \frac{1}{\sqrt 8}\bigl(5+\sqrt{ 9 - 8\kappa}\bigr)\sqrt{L} > \frac 13\Bigl(b_-^\eps+\eps^2+4\frac{C_L}{b_-^\eps+\eps^2}+2\eps\Bigr)
	$$
	and we get by continuity that
	\begin{align*}
		m(\eps, b_-^\eps+\eps^2,\mu^*_2) &\overset{\eps \to 0}{\longrightarrow} b_-^0+2\frac{C_L}{b_-^0}-\mu^* = \frac 23 \frac{(b_-^0)^2+C_L}{b_-^0}  = \sqrt {2L}.
	\end{align*}
	If $L<C_L$, $b_-^0$ satisfies $\frac 13 (b_-^0+4\frac{C_L}{b_-^0})<b_-^0+\frac{C_L}{b_-^0}$ since 
	$$
	\frac{C_L}{b_-^0} = \frac {1}{\sqrt 8} (3 \sqrt L +\sqrt{9L-8C_L})<\frac{1}{\sqrt 2}(3\sqrt L -\sqrt{9L-8 C_L})=2b_-^0.
	$$
	Thus, for sufficiently small $\eps>0$, $b=b_-^\eps+\eps^2$ and $a=b_-^\eps+\eps^2+2\frac{C_L}{b_-^\eps+\eps^2}-\mu^*_2+\eps$ satisfy $ab\ge C_L$.
	If $L=C_L$, we have $\mu^*_2=b_-^0+\frac{C_L}{b_-^0}$. Using
	$
	\frac{d}{d\eps}\big|_{\eps=0} \Bigl( b_-^\eps+\frac{C_L}{b_-^\eps}+\eps\Bigr) = \frac 12,
	$
	we get $\mu^*_2 \le b_-^\eps +\eps^2 + \frac{C_L}{b_-^\eps+\eps^2}+\eps$ for sufficiently small $\eps>0$ and, thus, $ab\ge C_L$ for the parameters $b=b_-^\eps+\eps^2$ and $a=b_-^\eps+\eps^2+2\frac{C_L}{b_-^\eps+\eps^2}-\mu^*_2+\eps$.
	Lastly, note that without loss of generality we can increase the Lipschitz constant $C_L$ as long as $\kappa<\frac 98$. We thus get that all friction parameters
	$$
	\mu \in \Bigl[\frac{1}{\sqrt 8}\bigl(5-\sqrt{ 9 - 8\kappa}\bigr)\sqrt{L}, \frac{1}{\sqrt 8}\bigl(5+\sqrt{ 9 - 8\kappa}\bigr)\sqrt{L}  \Bigr] \Big\backslash\Bigl\{\frac{5}{\sqrt 8}\sqrt{L}\Bigr\}
	$$
	give an optimal convergence rate of $\sqrt{2L}-\eps$. The case $\mu=\frac{5}{\sqrt 8}\sqrt{L}$ corresponds to $\kappa=\frac 98$ and will be treated below.
	
	Next, assume that $\kappa \ge \frac 98$ which implies that, for all $\eps > 0$, we have $\kappa^\eps \ge \frac{9}{8}$. 
	Using Lemma~\ref{lem:49378}, we get the convergence rate $m(\eps,b,\mu)=2(\mu-\frac{C_L}{b}-\eps)$ if $\mu \in (\frac{C_L}{b}+\frac{\eps}{2}, \mu_-^{\eps,b})$.	
	First, note that if $\mu_-^{0,b} \in (\frac{C_L}{b}+\frac{\eps}{2}, \mu_-^{\eps,b})$ we have
	$$
	m(\eps, b, \mu_-^{0,b})=  b+\frac{C_L}{b} - \sqrt{\Bigl(b+\frac{C_L}{b}\Bigr)^2-4L}-2\eps.
	$$
	Moreover, $m(\eps, b,\mu_-^{0,b}) \overset{b \to 0}{\longrightarrow}-2\eps$ and $m(\eps, b,\mu_-^0) \overset{b \to \infty}{\longrightarrow}-2\eps$ as well as
	$$
	\frac{d}{db} m(\eps,b,\mu_-^{0,b}) = \Bigl( 1-\frac{C_L}{b^2} \Bigr) \Bigl( 1-\frac{b+C_L/b}{\sqrt{(b+C_L/b)^2-4L}} \Bigr),
	$$
	so that	
	$\frac{d}{db}m(\eps,b,\mu_-^{0,b})>0$ for all $b< \sqrt{C_L}$ and $\frac{d}{db}m(\eps,b,\mu_-^{0,b})<0$ for all $b> \sqrt{C_L}$. Therefore, the maximal value for $m(\eps,b,\mu_-^{0,b})$ is attained at $b^*=\sqrt{C_L}$, where 
	$$
	\mu^*_3= \mu_-^{0,b^*} = 2\sqrt{C_L}-\sqrt{C_L-L} = (2\sqrt \kappa-\sqrt{\kappa-1})\sqrt L
	$$
	and 
	$$
	m(\eps, b^*,\mu^*_3)=2(\sqrt{C_L}-\sqrt{C_L-L}-\eps) \overset{\eps \to 0}{\longrightarrow} 2(\sqrt{\kappa}-\sqrt{\kappa-1})\sqrt L.
	$$
	Now, consider 
	\begin{align*}
		\mu_{-}^{\eps,b^*} &= \frac 12 \Bigl( 4 \sqrt{C_L}+\frac{3\eps}{2} - \sqrt{\bigl(2 \sqrt{C_L}+\frac{\eps}{2}\bigr)^2-4L} \,  \Bigr)
	\end{align*}
	as a function of $\eps$. Note that for all $\eps\ge 0$ we have $
	\frac{L}{C_L} \le \frac 29 \Bigl( 2+\frac{\eps}{2\sqrt{C_L}} \Bigr)^2,
	$ and, thus,
	\begin{align*}
		\frac{d}{d\eps} \mu_{-}^{\eps,b^*} &= \frac 12 \Bigl( \frac 32 - \frac 12 \Bigl(\bigl(2 \sqrt{C_L}+\frac{\eps}{2}\bigr)^2-4L\Bigr)^{-1/2}  \bigl(2 \sqrt{C_L}+\frac{\eps}{2}\bigr) \Bigr) \ge 0
	\end{align*}
	Therefore, for sufficiently small $\eps >0$ we have $\mu^*_3 \in (\frac{C_L}{b^*}+\frac{\eps}{2}, \mu_-^{\eps,b^*})$.
	Moreover, note that
	$
	\mu^* _3 < 2 \sqrt{C_L} = b^* + \frac{C_L}{b^*}
	$
	such that the parameters $b^*$ and $a=b^*+2\frac{C_L}{b^*}-\mu+\eps$ satisfy $ab^*\ge C_L$. Finally, for $\kappa = \frac 98$ we get $2(\sqrt{\kappa}-\sqrt{\kappa-1})\sqrt L= \sqrt{2L}$ and $\mu^*_3=\frac{5}{\sqrt 8}\sqrt L$.
\end{proof}

\begin{proof}[Proof of Corollary~\ref{cor:prob1}]
	Let $r>0$ such that $B_r(y) \subset \mathcal D$. Then, for every $r_0 < r$ we have $\{X_0 \in B_{r_0(y)}\} \subset \IA_0$.
	Choose $\bar \gamma, a,b, \eps >0$ as in Proposition~\ref{prop:convrate} and Lemma~\ref{lem:feasible1} such that $m=\min (a, \mu-a+b)>\eps$. Then, Proposition~\ref{prop:convrate} (i) states that there exists a constant $C(r_0)\ge 0$ such that  for all $n \in \N$
		$$
		\E[\1_{\IA_{n-1}}|V_n|^2]^{1/2} \le C(r_0) \exp\Bigl(-\frac 12 (m -\eps)t_n\Bigr).
		$$
		Note that, by Lemma~\ref{rem:Lipschitz}, the inverse PL-inequality \eqref{eq:inversePL} is satisfied for all $x \in B_{r/2}(y)$. Therefore, following \eqref{eq:Yconv2} together with \eqref{eq:238477} and \eqref{eq:2776362562} the constant $C(r_0)$ only depends on $\E[\1_{\IA_0}f(X_0)], \E[\1_{\IA_0} E_0]$ and $\E[\1_{\IA_0} |V_0|^2]$.		
		Now, using the fact that $f(y)=0$ and $\nabla f(y)=0$ one has $C(r_0)\to 0$ as $r_0 \to 0$.
	By Markov's inequality, we get for $r_0<r$
	\begin{align*}
		\P(\IA_\infty^c) &= \P \Bigl( \bigcup_{n=1}^\infty \IA_n^c \cap \IA_{n-1} \Bigr) \le \P \Bigl( \sup_{n \in \N} \1_{\IA_{n-1}} \sum_{i=1}^n \gamma_i |V_i| > r-r_0 \Bigr) \\
		& \le \frac{1}{r-r_0} \sum_{i=1}^\infty \gamma_i \E[  \1_{\IA_{i-1}}  |V_i|]\le \frac{C(r_0)}{r-r_0}  \int_0^{\infty}  \exp\Bigl(-\frac 12 ( m-\eps)t\Bigr) \, dt,
	\end{align*}
	and, thus, $\P(\IA_\infty^c)\to 0$ as $r_0 \to 0$.
\end{proof}

\section{Momentum stochastic gradient descent in continuous time} \label{sec:MSDE}
In this section, we study the diffusion process $(X_t)_{t \ge 0}$ defined in \eqref{eq:SDEintro}. 
We show that if the friction parameter is sufficiently large compared to the size of the stochastic noise we have almost sure exponential convergence of $(f(X_t))_{t \ge 0}$ for an objective function $f \in C^2$ that satisfies the PL-condition in an open set $\mathcal D \subset \R^d$.

\begin{theorem}\label{theoSDE1}
	Let $L, \sigma > 0$ and $\mathcal D \subset \R^d$ be an open set. Set $T:=\inf\{t \ge 0: X_t \notin \mathcal D\}$ and assume that for all $x \in \mathcal D$
	\begin{align} \label{eq:2863}
		|\grad f(x)|^2 \ge 2Lf(x)  \quad \text{ and } \quad  \|\Sigma(x)\|_F^2 \le \sigma f(x).
	\end{align}
	If $\mu>\frac{C_L\sigma}{4L}$ then:
	\begin{enumerate}
		\item[(i)] There exist $C,m>0$ such that for all $t \ge 0$
		$$
		\E[\1_{\{T>t\}}f(X_t)] \le C \exp(-mt).
		$$
		\item[(ii)] For all $m'<m$ we have $\exp(m't)f(X_t) \to 0$ almost surely on the event $\{T=\infty\}$.
		\item[(iii)] The process $(X_t)_{t \ge 0}$ converges almost surely on $\{T=\infty\}$.
	\end{enumerate}
\end{theorem}

In order to derive an explicit value for the convergence rate $m$, one has to solve a constrained optimization task. The exact formulation of the optimization task can be found in the statement of Lemma~\ref{lem:constrained1} below (see also Remark~\ref{rem:constrained2}).
Next, we give an estimate for the optimal choice of the friction parameter $\mu$ and the corresponding convergence rate.

\begin{theorem} \label{theoSDE2}
	Let $L,\sigma>0$.
	Define $C_L^* = C_L \vee \frac 98 L$, assume that $0<\sigma<4 \frac{L}{\sqrt{C_L^*}}$ and choose
	$$
	\mu = 2\sqrt{C_L^*} - \sqrt{C_L^* - L + \frac 14 \sqrt{C_L^*} \sigma}.
	$$
	Then, under the assumption \eqref{eq:2863} there exists a $C\ge 0$ such that
	$$
	\E[\1_{\{T>t\}} f(X_t)]\le C \exp(-mt),
	$$
	for all $t \ge 0$, where
	$$
	m= 2\Biggl(\sqrt{C_L^*} -\sqrt{C_L^*-L+\frac 14 \sqrt{C_L^*}\sigma}\Biggr).
	$$
\end{theorem}

\begin{remark} \label{rem:compSDE}
	In this remark, we compare the convergence rate for the continuous-in-time MSGD \eqref{eq:SDEintro} with the continuous-in-time counterpart for SGD, which is given by the SDE
	\begin{align} \label{eq:SDESGD}
		d\hat X_t = -  \grad f(\hat X_t ) \, d t +  \Sigma(\hat X_t) \, d W_t.
	\end{align}
	In the non-overparameterized setting, convergence rates for the SDE \eqref{eq:SDESGD} have been derived in~\cite{dereich2022cooling}. Following the arguments in~\cite{dereich2022cooling} and using the assumptions of Theorem~\ref{theoSDE2}, it is straightforward to show that
	$$
	\E[\1_{\{T>t\}}f(\hat X_t)] \le C \exp(-m_{\operatorname{SGD}}t)
	$$
	with rate $m_{\operatorname{SGD}}=2L-\frac 12 C_L \sigma$. Moreover, choosing the objective function $f(x)=\frac{L}{2} x^2$ shows that $\E[\1_{\{T>t\}}f(\hat X_t)]$ does not converge to zero if $\sigma > 4 \frac {L}{C_L}$. In contrast, the MSGD process \eqref{eq:SDEintro} converges exponentially to the set of critical points for all $\sigma \ge 0$ as long as the friction parameter satisfies $\mu>\frac{C_L \sigma }{4 L}$. The explicit rate of convergence for \eqref{eq:MSGDintro} is given as the solution of an optimization task over the friction parameter $\mu$, see Lemma~\ref{lem:constrained1} and Remark~\ref{rem:constrained2}. In Figure~\ref{fig:SDErate} below this optimization task is solved numerically for different values of $L,C_L$ and $\sigma$ using \emph{fminsearch} in \emph{Matlab}.
	
	We observe that continuous-in-time MSGD converges faster compared to continuous-in-time SGD in the case of large noise or convergence to flat minima, i.e. small $L$, while a large condition number $\kappa = \frac{C_L}{L}$ weakens this effect for small noise. 
	
	\begin{figure}[H] 
		\centering
		\begin{subfigure}{0.32\textwidth}
			\centering
			\includegraphics[width=\linewidth]{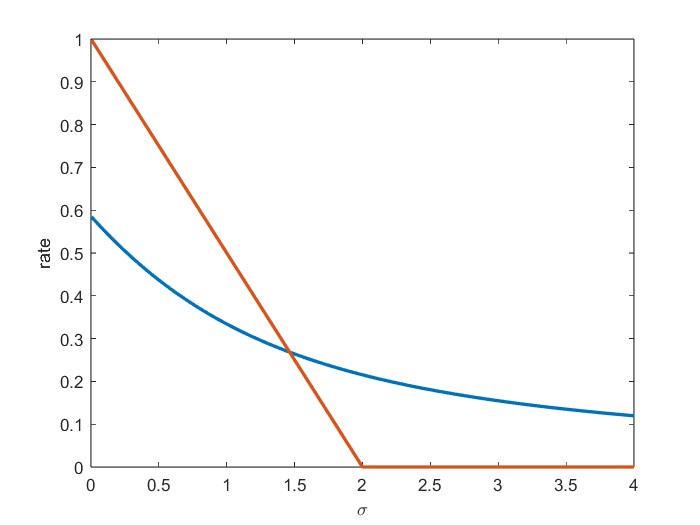}
			\caption{$L=0.5$, $C_L=1$}
			\label{fig:graph1}
		\end{subfigure}
		\hfill
		\begin{subfigure}{0.32\textwidth}
			\centering
			\includegraphics[width=\linewidth]{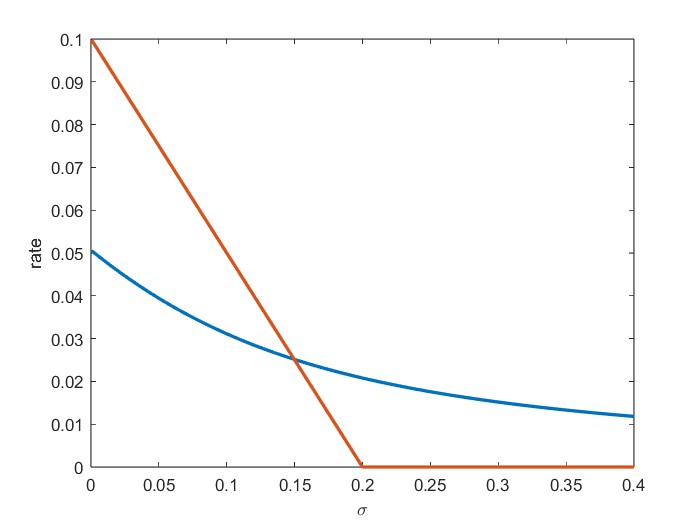}
			\caption{$L=0.005$, $C_L=1$}
			\label{fig:graph2}
		\end{subfigure}
		\hfill
		\begin{subfigure}{0.32\textwidth}
			\centering
			\includegraphics[width=\linewidth]{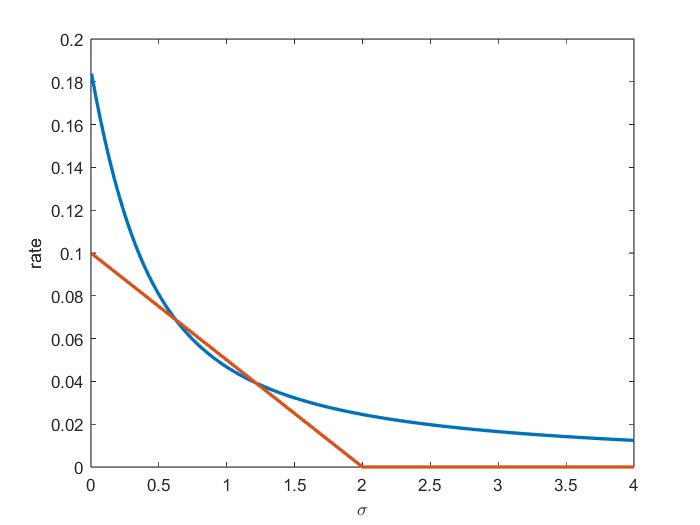}
			\caption{$L=0.05$, $C_L=0.1$}
			\label{fig:graph3}
		\end{subfigure}
		
		\caption{Comparison of the convergence rate $m$ for MSGD ({\blue blue}) and SGD ({\orange orange}) in continuous time in the sense of Theorem~\ref{theoSDE1} (i) depending on the noise intensity $\sigma$ ($x$-axis) for different values of $L$ and $C_L$.}
		\label{fig:SDErate}
	\end{figure}
\end{remark}

Theorem~\ref{theoSDE1} and Theorem~\ref{theoSDE2} are local analyses for the process $(X_t)_{t \ge 0}$ on the domain $\cD$, where $0$ is the only critical level and the stochastic noise vanishes as the objective function value approaches its minimum.\footnote{In this section, it is sufficient to assume that $0=\inf_{x \in \mathcal D} f(x)$.}
We can use estimates from the proof of Theorem~\ref{theoSDE1} to show that the expected length of the trajectory can be bounded by a constant that decays with the initial speed and the size of the gradient and value of the loss function at initialization. This allows us to bound the exit probability of the set $\mathcal D$. Thus, if we start the process $(X_t)_{t \ge 0}$ close to an optimal value in $\mathcal D$ and with small initial velocity $V_0$, $(X_t)_{t \ge 0}$ never hits the boundary of $\mathcal D$ and converges to a global minimum, with high probability.

\begin{corollary} \label{cor:probSDE}
	Let $y \in \cD$ with $f(y)=0$ and $\mu > \frac{C_L\sigma}{4L}$. Then, under the assumptions of Theorem~\ref{theoSDE1}, for every $\eps>0$ there exists an $r_0>0$ such that if $X_0 \in B_{r_0}(y)$, almost surely, and $\E[|V_0|^2]\le r_0$ we have that
	$$
	\P(T<\infty ) \le \eps.
	$$
\end{corollary}

\begin{remark}\label{rem:Lip2}
	Note that $C_L$ denotes the Lipschitz constant of $\grad f$. The precise value of the Lipschitz constant of $\Sigma$ does not appear in the statements of the results.
	We can weaken the assumptions on the Lipschitz continuity of $\grad f$ and $\Sigma$ in the following sense.
	Assume that $\grad f$ and $\Sigma$ are only Lipschitz continuous on $\mathcal D$. 
	Then, there exists a continuous semimartingale $(X_t,V_t)_{t \ge 0}$ satisfying \eqref{eq:SDEintro} up to the stopping time $T=\inf\{t \ge 0: X_t \notin \mathcal D\}$. Now, in order to derive the statements of Theorem~\ref{theoSDE1} and Theorem~\ref{theoSDE2} it is sufficient to assume $\inf_{x \in \overline{\mathcal D}}f(x)=0$ and, for all $x \in \mathcal D$,
	$$
	|\grad f(x)|^2\le 2 C_L f(X_t),
	$$
	where $C_L$ denotes the Lipschitz constant of $\grad f$ on $\mathcal D$. Lemma~\ref{rem:Lipschitz} shows how the latter inequality follows from Lipschitz continuity of $\grad f$ on a larger domain. For the statement of Theorem~\ref{theo1} (ii) one additionally needs Lipschitz continuity of $\grad f$ and $\Sigma$ on a convex set containing $\mathcal D$.
\end{remark}

We start proving the main results of this section. The following proposition gives exponential convergence for the expectation of the objective function value under technical assumptions on the parameters $\mu,C_L,L$ and $\sigma$. Again, the proofs are based on the random Lyapunov function $(E_t)_{t \ge 0}$ defined by
\begin{align*} 
	E_t = a f(X_t) + \langle \grad f(X_t), V_t \rangle + \frac{b}{2} |V_t|^2.
\end{align*}
Note that $(E_t)_{t \ge 0}$ is a continuous, integrable process. If $(X_t)_{t \ge 0}$ is able to leave $\mathcal D$ we have to make sure that the Lyapunov function is non-negative at the exit time which is satisfied if $ab \ge C_L$, see Lemma~\ref{lem:Epositive}.

\begin{proposition} \label{prop:SDErate}
	Let $L, \sigma > 0$. Let $T$ be an $(\cF_t)_{t \ge 0}$-stopping time such that for all $t\ge 0$ on $\{T>t\}$
	\begin{align} \label{eq:PL}
		2Lf(X_t) \le |\grad f(X_t)|^2 \quad \text{ and } \quad  \|\Sigma(X_t)\|_F^2 \le \sigma f(X_t).
	\end{align}
	Furthermore, let $a,b >0 $ and suppose that
	\begin{align} \label{eq:128188}
		\mu-a+b > 0 \; , \; \frac{b}{2}\sigma-a^2+a\mu+ab- 2L \le 0 \; \text{and} \;   C_L-\frac b2 (\mu+a-b) \le 0 .
	\end{align}
	If $\P(T = \infty)<\P(T>0)$ additionally assume that $ab\ge C_L$.
	Then:
	\begin{enumerate}
		\item[(i)] There exist a constant $C>0$ such that
		\begin{align*}
			\max(\E[\1_{\{T>t\}} f(X_t)], \E[&\1_{\{T>t\}}|V_t|^2])\\
			&\le 
			\begin{cases}
				C \exp(-mt), & \text{ if } a \neq \mu-a+b, \\
				C(1+ t) \exp(-mt), & \text{ if } a = \mu-a+b,
			\end{cases} 
		\end{align*}
		for all $t\ge 0$, where $m=\min(a,\mu-a+b)$.
		\item[(ii)] $(X_t)_{t \ge 0}$ converges almost surely on $\{T=\infty\}$.
	\end{enumerate}
\end{proposition}

\begin{proof}
	(i): 
	First, we show the exponential convergence of $(\E[\1_{\{T>t\}} E_t])_{t \ge 0}$ and, afterwards, we show that this implies (i). 
	
	Since $(X_t)_{t \ge 0}$ is of bounded variation we get by It\^{o}'s formula that
	\begin{align*}
		df(X_t) = \langle \grad f(X_t), V_t \rangle dt \quad \text{ and } \quad d |V_t|^2 = 2 \langle V_t , dV_t\rangle  +   \|\Sigma(X_t)\|_F^2 dt,
	\end{align*}
	and by It\^{o}'s product rule that
	$$
	d \langle \grad f(X_t), V_t \rangle = \langle \nabla f(X_t), dV_t \rangle  + \langle V_t, \Hess f(X_t) V_t \rangle dt.
	$$	
	Thus, for $(\tilde E_t)_{t \ge 0}= (\1_{\{T > t\}} E_t)_{t \ge 0}$ we have
	\begin{align*}
		d \tilde E_t = & \1_{\{T>t\}}\Bigl((a-\mu-b) \langle \grad f(X_t), V_t \rangle - |\grad f(X_t)|^2 - b \mu |V_t|^2 + \langle V_t, \Hess f(X_t) V_t \rangle \\
		&+ \frac{b}{2}  \|\Sigma(X_t)\|_F^2 \Bigr) dt + dM_t - d\xi_t,
	\end{align*}
	where $(M_t)_{t \ge 0}$ denotes the $L^2$-martingale
	$$
	(M_t)_{t \ge 0} = \Bigl(  \int_0^{T\wedge t}\langle \grad f(X_u) + b V_u, \Sigma(X_u) dW_u \rangle \Bigr)_{t \ge 0},
	$$
	and $(\xi_t)_{t \ge 0}$ denotes the (almost surely) non-negative and increasing process given by
	$$
	\xi_t :=  \begin{cases}
		0, & \text{ if } t<T \text{ or } T=0, \\
		E_T, & \text{otherwise.}
	\end{cases}
	$$
	Using \eqref{eq:PL} and the Lipschitz continuity of $\grad f$, we get, for all $0 \le s < t $,
	\begin{align*}
		\tilde E_t-\tilde E_s \le & \int_{s \wedge T}^{t \wedge T} (a-\mu-b) \langle \grad f(X_u), V_u \rangle - |\grad f(X_u)|^2 - (b \mu - C_L ) |V_u|^2 \,  du \\
		& +\int_{s \wedge T}^{t \wedge T} \frac{b}{2}\sigma  f(X_u) \,  du+ M_t-M_s -(\xi_t-\xi_s).
	\end{align*}
	By definition of $(E_t)_{t \ge 0}$, we have, for all $u \ge 0$,
	$$
	\langle \grad f(X_u), V_u \rangle =  E_u - a f(X_u)  -\frac{b}{2} |V_u|^2,
	$$
	so that, using the PL-inequality \eqref{eq:PL}, we get
	\begin{align*}
		\tilde E_t-\tilde E_s \le  & \int_{s \wedge T}^{t \wedge T} (a-\mu-b) \tilde E_u - \Bigl(\frac b2 \mu - C_L+\frac{b}{2}a  -\frac{b^2}{2} \Bigr) |V_u|^2 \,  du \\
		&+\int_{s \wedge T}^{t \wedge T}\Bigl(\frac{b}{2}\sigma -a^2+a\mu+ab-2L\Bigr) f(X_u) \, ds + M_t-M_s -(\xi_t-\xi_s).
	\end{align*}
	With the dominated convergence theorem $(e_t)_{t \ge 0}:=(\E[\1_{\{T>t\}} E_t])_{t \ge 0}$ is lower semicontinuous such that using \eqref{eq:128188} and Proposition~2.3 in~\cite{matusik2020finite} we have $e_t \le e_0 \exp((a-\mu-b)t)$, for all $t \ge 0$.

	Next, we use the estimates for $(e_t)_{t \ge 0}$ in order to derive a rate of convergence for $\varphi_t=\E[\1_{\{T > t \}}f(X_t)]$. Recall that
	$$
	df(X_t) = \langle \grad f(X_t), V_t \rangle \, dt = E_t dt - a f(X_t) dt -\frac{b}{2} |V_t|^2 dt.
	$$
	Thus, $(\tilde f_t)_{t \ge 0} := (\1_{\{T > t \}}f(X_t))_{t \ge 0}$ is a non-negative process that satisfies
	$$
	d\tilde f_t = \1_{\{T>t\}} \bigl( E_t-af(X_t)-\frac b2 |V_t|^2\bigr) dt - d\zeta_t,
	$$
	where $(\zeta_t)_{t \ge 0}$ is a non-negative, increasing process given by
	$$
	\zeta_t :=  \begin{cases} 0, & \text{ if }t<T \text{ or } T=0 \\
		f(X_T),& \text{otherwise.}
	\end{cases} 
	$$
	Taking expectation, we note that $(\varphi_t)_{t \ge 0}$ is lower semicontinuous and, for all $0\le s < t$, we have
	$$
	\varphi_t-\varphi_s \le \E\Bigl[\int_s^t \1_{\{T<u\}} (E_u - a f(X_u)) \, du \Bigr]= \int_s^t (e_u - a \varphi_u) \, du.
	$$
	Using Proposition~2.3 in~\cite{matusik2020finite} we get for all $t \ge 0$ that 
	\begin{align*}
		\varphi_t &\le \varphi_0 \exp(-a t) + \int_0^t \exp(a (s-t)) e_s ds \\
		&= \varphi_0 \exp(-at) + e_0 \exp(-at) \int_0^t \exp((2a-\mu-b) s) ds. 
	\end{align*} 
	\begin{align*}
		\varphi_t \le \varphi_0 \exp(-at) + e_0  \Bigl( \frac{1}{2a-\mu-b} \bigl(\exp((a-\mu-b) t)-\exp(-at)\bigr) \Bigr)
	\end{align*}
	Conversely, for $2a-\mu-b=0$ we get
	\begin{align*}
		\varphi_t &\le (\varphi_0+e_0 t) \exp(-at).
	\end{align*}
	
	Regarding the convergence of $(\E[\1_{\{T>t\}} |V_t|^2])_{t \ge 0}$ note that since $f(X_t)\ge 0$
	\begin{align*}
		\frac b2 |V_t|^2 &\le E_t - \langle \grad f(X_t), V_t \rangle \le E_t + |\grad f(X_t)| \, |V_t|,
	\end{align*}
	which, analogously to \eqref{eq:Yconv2} implies 
	\begin{align}  \label{eq:2374278624}
		\frac b4\E[\1_{\{T>t\}}|V_t|^2] \le e_t + \frac 1b \varphi_t.
	\end{align}
	By the computations above, there exists a constant $C\ge 0$ such that 
	$$
	\E[\1_{\{T>t\}}|V_t|^2]\le 	\begin{cases}
		C \exp(-mt), & \text{ if } a \neq \mu-a+b, \\
		C(1+ t) \exp(-mt), & \text{ if } a = \mu-a+b.
	\end{cases} 
	$$
	
	(ii):
	Using (i), we get 
	\begin{align*}
		\E\Bigl[\int_0^T |V_s|\,  ds\Bigl] \le \int_0^\infty \E[\1_{\{T>s\}}|V_s|^2]^{1/2} \, ds < \infty
	\end{align*}
	such that $\int_0^T |V_s|\,  ds$ is almost surely finite. Since, 
	$|X_{t}-X_{s}|\le \int_{s}^{t} |V_u|\,  du$, for all $0\le s \le t$, $(X_t)_{t \ge 0}$ converges almost surely on $\{T=\infty\}$.
\end{proof}

The next lemma shows that, if the friction is sufficiently large compared to the size of the stochastic noise, we may find parameters $a,b >0$ such that Proposition~\ref{prop:SDErate} applies.

\begin{lemma} \label{lem:2847}
	Let $L,\sigma >0$. Then, for all $\mu >\frac{C_L \sigma}{4 L}$ there exist $a,b >0$ such that \eqref{eq:128188} holds and $ab \ge C_L$. 
\end{lemma}

\begin{proof}
	Let $b>0$ and choose $a=b+2 \frac{C_L}{b}-\mu$. Note that $a>0$ iff $\mu < b+2\frac{C_L}{b}$ and $a-\mu-b < 0$ iff $\mu > \frac {C_L}{b}$. Now, $\frac b2 \mu - C_L+\frac{b}{2}a  -\frac{b^2}{2}=0$ and 
	\begin{align*}
		\frac{b}{2}\sigma-a^2+a\mu+ab- 2L= -2 \mu^2+\Bigl( 2b+\frac{6 C_L}{b} \Bigr) \mu - 4 \frac{C_L^2}{b^2} - 2 (C_L+L)+\frac{b}{2} \sigma.
	\end{align*}
	The right-hand side of the latter equation is a quadratic function that is only positive between the roots
	\begin{align} \label{eq:12341}	
		\mu_\pm^b = \frac 12 \Bigl( b+\frac{3C_L}{b} \pm \sqrt{\Bigl(b+\frac{C_L}{b}\Bigr)^2 - 4 L + b \sigma} \Bigr). 
	\end{align}
	Note that, for $b<\frac{4L}{\sigma}$ we have that $\mu_-^b > \frac{C_L}{b}$ and $\mu_+^b < b+2\frac {C_L}{b}$. Moreover, the assumption $ab\ge C_L$ is satisfied iff $\mu \le b+\frac {C_L}{b}$. Thus, the set of friction parameters $\mu$ that satisfy \eqref{eq:128188} for the given pair $(a,b)$ is equal to
	$$
	\Bigl(\frac{C_L}{b}, \mu_-^b\Bigr] \cup \Bigl[\mu_+^b, b+2\frac{C_L}{b}\Bigr)
	$$
	and the set of friction parameters $\mu$ that, additionally, satisfy  $ab \ge C_L$ for the given pair $(a,b)$ is contained in 
	$
	(\sfrac{C_L}{b}, \mu_-^b \wedge (b+\sfrac{C_L}{b})).
	$
	Note that, for all $0<b<\frac{4L}{\sigma}$, the latter interval is non-empty, the upper and lower bounds are continuous in $b$ and the lower bound satisfies 
	$$
	\frac{C_L}{b} \overset{b \to 0}{\longrightarrow} \infty \quad \text{ and } \quad \frac{C_L}{b} \overset{b \to 4L/\sigma}{\longrightarrow} \frac{C_L \sigma}{4 L}.
	$$
	We thus showed that for every $\mu>\frac{C_L\sigma}{4L}$ there exists a pair $(a,b)$ such that \eqref{eq:128188} is satisfied and $ab \ge C_L$.
\end{proof}

We are now in the position to prove Theorem~\ref{theoSDE1}. The second part of the proof is more involved compared to the corresponding result in discrete time since we cannot immediately use the Borel-Cantelli lemma. In an additional step, we have to show that the process does not deviate too much from the values it takes at discrete times. 

\begin{proof}[Proof of Theorem~\ref{theoSDE1}]
	(i) and (iii): 
	Clearly, $T$ is an $(\cF_t)_{t \ge 0}$-stopping time satisfying the assumption of Proposition~\ref{prop:SDErate}.
	By Lemma~\ref{lem:2847}, there exist parameters $a, b>0$ such that \eqref{eq:128188} holds and $ab \ge C_L$. We let $m= \min(a,\mu-a+b)$ if $a\neq \mu-a+b$ and $m\in(a,\infty)$, otherwise. Then, Proposition~\ref{prop:SDErate} implies that there exists a $C>0$ such that for all $t \ge 0$
	$$
	\E[\1_{\{T>t\}} f(X_t)]\le C \exp(-mt).
	$$
	and $(X_t)_{t \ge 0}$ converges almost surely on the event $\{T=\infty\}$.
	
	(ii): We denote $C_L' := \|\grad f\|_{\mathrm{Lip}(\R^d)}\vee \|\Sigma\|_{F,\mathrm{Lip}(\R^d)}$, where $\|\cdot\|_{F, \mathrm{Lip}(\R^d)}$ is the Lipschitz norm that is induced by the Frobenius norm. 
	Let $n \in \N_0$ and note that for $t \in [n,n+1]$
	$$
	\sup_{s \in [n,t]} \1_{\{T > s\}} |X_{s}-X_n|^2 \le \int_{n\wedge T}^{t\wedge T} |V_u|^2 \, du \le  (t-n) \sup_{s \in {[n,t]}} \1_{\{T >s\}} |V_{s}|^2 .
	$$
	Now,
	\begin{align*}
		\E\Bigl[ \sup_{s \in [n,t]} \1_{\{T>s\}} |V_s|^2 \Bigr]&\le 4 \Bigl( \E[ \1_{\{T>n\}} |V_n|^2 ] + \mu^2 \E\Bigl[ \int_n^{t} \sup_{u \in [n,s]} \1_{\{T>u\}} |V_u|^2 \, ds \Bigr] \\
		& \quad  +  \E\Bigl[ \int_{n \wedge T}^{t \wedge T}  |\grad f(X_u)|^2 \, du \Bigr] + \E\Bigl[ \sup_{s \in [n, t]} \Bigl| \int_{n \wedge T}^{s \wedge T}  \Sigma(X_u) \, dW_u \Bigr|^2 \Bigr] \Bigr).
	\end{align*}
	Using the Lipschitz continuity of $\grad f$, we get
	\begin{align*}
		\E\Bigl[ \int_{n \wedge T}^{t \wedge T}  |\grad f(X_u)|^2 \, du \Bigr]
		&\le 2\E[ \1_{\{T>n\}} |\grad f(X_n)|^2] + 2(C_L')^2\E\Bigl[ \int_{n}^{t} \sup_{s \in {[n,u]}} \1_{\{T>s\}} |V_{s}|^2 \, du \Bigr].
	\end{align*}
	Moreover, using Doob's $L^2$-inequality and the It\^{o}-isometry,
	\begin{align*}
		\E\Bigl[ \sup_{s \in [n, t]} \Bigl| \int_{n \wedge T}^{s \wedge T}  &\Sigma(X_u) \, dW_u \Bigr|^2 \Bigr]  \le 4 \E\Bigl[ \int_{n \wedge T}^{t\wedge T} \|\Sigma(X_u)\|_F^2  \, du \Bigr] \\
		&\le 8 \E[\1_{\{T>n\}}\|\Sigma(X_n)\|_F^2] + 8(C_L')^2 \E\Bigl[ \int_n^{t} \sup_{s \in {[n,u]}} \1_{\{T >s\}} |V_{s}|^2 \, du \Bigr].
	\end{align*}
	Hence,
	\begin{align*}
		\E\Bigl[ \sup_{s \in [n,t]} \1_{\{T>s\}} |V_s|^2 \Bigr]&\le 32 \Bigl( \E[ \1_{\{T>n\}} (|V_n|^2+|\grad f(X_n)|^2+\|\Sigma(X_n)\|_F^2) ] \\
		&+ (\mu^2+2(C_L')^2)  \int_n^{t} \E\Bigl[ \sup_{s \in [n,u]} \1_{\{T>s\}} |V_s|^2  \Bigr] \, du\Bigr).
	\end{align*}
	Thus, by Gronwall's inequality there exists a constant $C\ge 0$ such that for all $n \in \N_0$ and $n\le t \le n+1$ we have 
	$$
	\E\Bigl[ \sup_{s \in [n,t]} \1_{\{T>s\}} |V_s|^2 \Bigr] \le C \,  \E[ \1_{\{T>n\}} (|V_n|^2+|\grad f(X_n)|^2+\|\Sigma(X_n)\|_F^2) ].
	$$
	Using Proposition~\ref{prop:SDErate} (i), Lemma~\ref{rem:Lipschitz} and \eqref{eq:PL}, there exist a constants $C, > 0$ such that for all $n \in \N_0$
	\begin{align*}
		\E[ \1_{\{T>n\}} (|V_n|^2+|\grad f(X_n)|^2+\|\Sigma(X_n)\|_F^2) ] \le C \exp(-mn).
	\end{align*}
	Therefore, by the Lipschitz-continuity of $\grad f$,
	\begin{align*}
		\E\Bigl[ &\sup_{s \in [n,n+1]} \1_{\{T>s\}} |f(X_s)-f(X_n)| \Bigr] \le \E\Bigl[  \int_{n \wedge T}^{(n+1)\wedge T} |\grad f(X_u)|\,  |V_u| \, du \Bigr] \\
		&\le \E[\1_{\{T>n\}}|\grad f(X_n)|^2]^{1/2}  \E\Bigl[\sup_{s \in [n,n+1]} \1_{\{T>s\}}|V_s|^2\Bigr]^{1/2}\\
		& \quad +C_L' \E\Bigl[\sup_{s \in [n,n+1]} \1_{\{T>s\}}|V_s|^2\Bigr] \\
		&\le C \exp(-mn),
	\end{align*}
	for a constant $C>0$.
	
	Next, we prove the statement. 
	For $m'<m$, $\eps>0$ and $n \in \N$ consider the set
	$$
	\IB_n = \{T=\infty\} \cap \Bigl\{\sup_{t \ge n} \exp(m't)f(X_t)\ge \eps\Bigr\}.
	$$
	We use the Markov inequality, the Lipschitz continuity of $\grad f$ and (i) to get, for some $C \ge 0$, that 
	\begin{align*}
		\P(\IB_n) &\le \sum_{i=n}^\infty \P(\{T\ge i\} \cap \{\exp(m' (i+1)) f(X_i) \ge \eps/4\}) \\
		&+ \sum_{i=n}^\infty \P\Bigl(\{T\ge i+1\} \cap \Bigl\{ \sup\limits_{t \in [i,i+1]}\exp(m' (i+1)) (f(X_t)-f(X_i)) \ge \eps/4\Bigr\}\Bigr)\\
		& \le \sum_{i=n}^\infty \exp(m' (i+1))\frac{4}{\eps} \E[\1_{\{T \ge i\}} f(X_i)]\\
		& + \sum_{i=n}^\infty \exp(m' (i+1))\frac{4}{\eps} \E\Bigl[\sup_{t \in [i,i+1]}\1_{\{T \ge t\}} |f(X_t)-f(X_i)|\Bigr] \\
		& \le C \sum_{i=n}^\infty \exp((m'-m) i) \overset{n \to \infty}{\longrightarrow} 0.
	\end{align*}
	Hence,
	$$
	\P\Bigl(\{T=\infty\} \cap \Bigl\{\limsup_{t \to \infty} \exp(mt)f(X_t) \ge \eps\Big\}\Bigr) \le \P\Bigl(\bigcap_{n \in \N} \IB_n\Bigr) =0
	$$
	so that $\exp(m t)f(X_t)\to 0$ almost surely on $\{T=\infty\}$. 
\end{proof}

For the admissible friction parameters $\mu>\frac{C_L\sigma}{4L}$, we use Proposition~\ref{prop:SDErate} to yield a rate of convergence for the expected objective function value in dependency of the technical parameter $b$. In order to get an optimal value for the convergence rate, we optimize $m(b,\mu)$ defined in the following lemma over all admissible choices of $b$. 

\begin{lemma} \label{lem:constrained1}
	Let $L, \sigma > 0$. Let $T$ be an $(\cF_t)_{t \ge 0}$-stopping time such that for all $t\ge 0$
	\begin{align*} 
		2Lf(X_t) \le |\grad f(X_t)|^2 \quad \text{ and } \quad  \|\Sigma(X_t)\|_F^2 \le \sigma f(X_t), \quad \text{ on } \{T>t\}.
	\end{align*}
	Let $0<b<\frac{4L}{\sigma}$ and
	$
	\mu \in (\frac{C_L}{b}, \mu_-^b] \cup [\mu_+^b, b+2\frac{C_L}{b}),
	$
	where $\mu_{\pm}^b$ is given by \eqref{eq:12341}.
	If $\P(T = \infty)<\P(T>0)$ additionally assume that $\mu \le b+\frac{C_L}{b}$. Then, for all $\eps>0$ there exists a $C>0$ such that for all $t \ge 0$
	\begin{align} \label{eq:28362}
		\E[\1_{\{T>t\}} f(X_t)] \le C \exp(-m(b,\mu) t),
	\end{align}
	where
	\begin{align} \label{eq:238565242}
		\begin{split}
			m(b,\mu) = \begin{cases}
				2(\mu-\frac{C_L}{b}), & \text{ if } \mu < \frac{1}{3} (b+4\frac{C_L}{b}) \\
				b+2\frac{C_L}{b}-\mu, & \text{ if } \mu > \frac{1}{3} (b+4\frac{C_L}{b}) \\
				\frac 23(b+\frac{C_L}{b})-\eps & \text{ if } \mu = \frac{1}{3} (b+4\frac{C_L}{b}).
			\end{cases}
		\end{split}
	\end{align}
\end{lemma}

\begin{proof}
	Let $b$ and $\mu$ be as in the assumptions and set $a=b+2\frac{C_L}{b}-\mu$. Then, by Proposition~\ref{prop:SDErate} and Lemma~\ref{lem:2847}, we get exponential convergence of $(\E[\1_{\{T>t\}}f(X_t)])_{t\ge 0}$ with rate $m(b,\mu)$, where for $a \neq \mu-a+b$ we have
	$$
	m(b,\mu)=\min(a,\mu-a+b)=\min\Bigl(b+2\frac{C_L}{b}-\mu,2\Bigl(\mu-\frac{C_L}{b}\Bigr)\Bigr)
	$$
	and, otherwise, for all $\eps >0$, we can choose $m(b,\mu):=a-\eps$.
\end{proof}

In order to derive the optimal rate of convergence for arbitrary, fixed friction parameter $\mu>\frac{C_L \sigma}{4L}$, one would have to maximize $m(b,\mu)$ over all admissible parameters $b$. We proceed by optimizing over $\mu$ and $b$, simultaneously.

\begin{remark} \label{rem:constrained2}
	We optimize the rate $m(b,\mu)$ over the admissible choices of $b$ and $\mu$. First, we fix $b>0$ and note that, in the case $m(b,\mu)=2(\mu-\frac{C_L}{b})$, the rate is optimal for the largest admissible $\mu$. If $\Phi(b):=b^4+\frac 98 \sigma b^3+\Bigl(2C_L-\frac{9}{2}L\Bigr)b^2+C_L^2>0$ we have $\mu_-^b<\frac{1}{3} (b+4\frac{C_L}{b})$. Thus, taking $\mu=\mu_-^b$ we get
	$$
	m(b,\mu_-^b) = b+\frac{C_L}{b} -\sqrt{\Bigl(b+\frac{C_L}{b}\Bigr)^2-4L+b\sigma}.
	$$
	If $\Phi(b) \le 0$ we have $\mu_-^b \ge \frac{1}{3} (b+4\frac{C_L}{b})$. Thus, taking $\mu=\frac 13 (b+4\frac{C_L}{b})$ we get, for any $\eps >0$,
	$$
	m\Bigl(b,\frac 13\Bigl(b+4\frac{C_L}{b}\Bigr)\Bigr) = \frac 23 \Bigl(b+\frac{C_L}{b}\Bigr)-\eps.
	$$
	In the case $m(b,\mu)=b+2\frac{C_L}{b}-\mu$, the rate is optimal for the smallest admissible $\mu$. If $\Phi(b) \ge 0$ we take $\mu=\mu_+^b$ and get
	$$
	m(b,\mu_+^b) = \frac 12 \Biggl( b+\frac{C_L}{b} -\sqrt{\Bigl(b+\frac{C_L}{b}\Bigr)^2-4L+b\sigma} \Biggr).
	$$
	If $\Phi(b)<0$ we take $\mu=\frac 13(b+4\frac{C_L}{b})$ and get, for any $\eps>0$,
	$$
	m\Bigl(b,\frac 13 \Bigl(b+4\frac{C_L}{b}\Bigr)\Bigr) = \frac 23 \Bigl(b+\frac{C_L}{b}\Bigr)-\eps.
	$$
\end{remark}

\begin{proof}[Proof of Theorem~\ref{theoSDE2}]
	Note that $C_L^*\ge C_L$ such that $C_L^*$ is a Lipschitz constant for $\grad f$ as well. By the computations above, the assumptions of Proposition~\ref{prop:SDErate} are satisfies for $C_L$ replaced by $C_L^*$, $b=\sqrt{C_L^*}$, $a=3\sqrt{C_L^*}-\mu$ and $\mu = 2\sqrt{C_L^*} - \sqrt{C_L^* - L + \frac 14 \sqrt{C_L^*} \sigma} $. In particular, $ab \ge C_L^*$, since $\mu \le b+\frac{C_L^*}{b}$, and $\Phi(b)>0$, since $\sigma >0$. Thus, there exists a constant $C>0$ such that
	$$
	\E[\1_{\{T>t\}}f(X_t)] \le C \exp(-mt),
	$$
	where 
	$
	m = 2 (\mu-\frac{C_L^*}{b})= 2\sqrt{C_L^*} -\sqrt{4C_L^*-4L+\sqrt{C_L^*}\sigma}.
	$
\end{proof}

\begin{proof}[Proof of Corollary~\ref{cor:probSDE}]
	Let $r_0$ be sufficiently small such that $B_{r_0}(y) \subset \mathcal D$. Let $r_1<r_0$, $\tilde T'_{r_1}=\inf\{t \ge 0: \int_0^t |V_s| \, ds >r_0-r_1\}$ and note that $T'_{r_1} < T$, almost surely. Thus, we get by \eqref{eq:2374278624},
	\begin{align*}
		\P(T<\infty) &\le \P \Bigl( \int_0^{T} |V_s| \, ds \ge  r_0-r_1 \Bigr) \le \frac{1}{r_0-r_1} \E\Bigl[ \int_0^{T} |V_s| \, ds  \Bigr] \\
		& \le \frac{1}{ r_0-r_1} C(e_0,\varphi_0) \int_0^\infty \exp(-ms) \, ds,
	\end{align*}
	for an $m>0$ and a constant $C(e_0,\varphi_0)$ that only depends on $e_0= \E[\1_{\{T>0\}} E_0]$ and $\varphi_0= \E[\1_{\{T>0\}} f(X_0)]$ and satisfies $C(e_0,\varphi_0) \to 0$ as $(e_0,\varphi_0) \to 0$.
	Thus, for every $\eps>0$ there exists an $r_0>0$ such that $\P(T<\infty) \le \eps $.
\end{proof}

\subsection*{Acknowledgements}
The authors would like to thank Vitalii Konarovskyi for carefully proof-reading the article and his many valuable suggestions. Thanks to Benjamin Fehrman for fruitful discussions in the beginning of this project.
The authors were supported by the Deutsche Forschungsgemeinschaft (DFG, German Research Foundation) – SFB 1283/2 2021 – 317210226.   BG acknowledges support by the Max Planck Society through the Research Group "Stochastic Analysis in the Sciences (SAiS)".


\newcommand{\etalchar}[1]{$^{#1}$}

\end{document}